\documentclass[english]{amsart}
\usepackage[latin1]{inputenc}
\usepackage{youngtab}
\usepackage{amsthm,amsmath,amssymb}
\usepackage{youngtab}
\theoremstyle{plain}
\newtheorem{thm}{Theorem}[section]

\newtheorem{prop}[thm]{Proposition}
\newtheorem{lem}[thm]{Lemma}

\theoremstyle{definition}
\newtheorem*{defn}{Definition}

\newtheorem*{notaz}{Notation}
\newtheorem{exa}[thm]{Example}
\theoremstyle{remark}
\newtheorem*{remark}{Remark}
\newcommand{\Sh}{\mathrm{Sh}}
\newcommand{\Fer}{\mathrm{Fer}}
\newcommand{\eqdef}{\stackrel{\rm def}{=}}
\usepackage{babel}

\newcommand{\St}{\mathcal {ST}}
\newcommand{\qc}{G(r,p,q,n)}

\newcommand{\Irr}{\mathrm{Irr}}
\newcommand{\tr}{\mathrm{tr}}
\newcommand{\GCD}{\mathrm{GCD}}
\newcommand{\Inv}{\mathrm{Inv}}
\newcommand{\Id}{\mathrm{Id}}
\newcommand{\inv}{\mathrm{inv}}
\newcommand{\Pair}{\mathrm{Pair}}
\newcommand{\pair}{\mathrm{pair}}

\newcommand{\Supp}{\mathrm{Supp}}
\newcommand{\tho}{\mathrm{th}}

\newcommand{\sign}{\mathrm{sign}}

\newcommand{\Ind}{\mathrm{Ind}}
\newcommand{\GL}{\mathrm{GL}}
\newcommand{\Stab}{\mathrm{Stab}}
\newcommand{\stab}{\mathrm{stab}}

\author{Fabrizio Caselli}\author {Roberta Fulci}
\title{Gelfand models and Robinson-Schensted correspondence}

\begin{document}

\begin{abstract}In [F. Caselli, Involutory reflection groups and their models, J. Algebra 24 (2010), 370--393] there is constructed a uniform Gelfand model for all non-exceptional irreducible complex reflection groups which are involutory. Such model can be naturally decomposed into the direct sum of submodules indexed by $S_n$-conjugacy classes, and we present here a general result that relates the irreducible decomposition of these submodules with the projective Robinson-Schensted correspondence. This description also reflects in a very explicit way the existence of split representations for these groups.
\end{abstract}
\maketitle
\section[intro]{Introduction}

Given a finite dimensional vector space $V$ on the complex field, $G<GL(V)$ is a reflection group if it is generated by reflections, i.e. elements of finite order fixing a hyperplane of $V$ pointwise. Finite irreducible complex reflection groups were completely classified in the fifties \cite{ST} by Shephard and Todd. They consist of an infinite family $G(r,p,n)$, where $r,p,n\in \mathbb{N}$ and  $p|r$, which are the main subject of this paper, and 34 more sporadic groups.

This work finds its roots in the introduction of a new family of groups, called \textit{projective reflection groups} \cite{Ca1}. They can be roughly described as quotients - modulo a scalar group - of finite reflection groups. If we quotient a group $G(r,p,n)$ modulo the cyclic scalar subgroup $C_q$, we find a new group $G(r,p,q,n)$, so that in this notation we have $G(r,p,n)=G(r,p,1,n)$. We define the \textit{dual group}  $G(r,p,q,n)^*$ as the group $G(r,q,p,n)$ obtained by simply exchanging the parameters $p$ and $q$. It turns out that many objects related to the algebraic structure of a projective reflection group $G$ can be naturally described by means of the combinatorics of its dual $G^*$ (see \cite{Ca1,Ca2}). For example, its representations.

A Gelfand model of a finite group $G$ is a $G$-module isomorphic to the multiplicity-free sum of all the irreducible complex representations of $G$. The study of Gelfand models originated from \cite{BGG} and has found a wide interest in the case of reflection groups and other related groups (see, e.g.,\cite{K2,IRS, B,APR}). 
In \cite{Ca2}, a Gelfand model $(M,\varrho)$ was constructed (relying on the concept of duality in an essential way) for every \textit{involutory} projective reflection group $G(r,p,q,n)$. 

A finite subgroup of $GL(V)$ is involutory if the number of its absolute involutions, i.e. elements $g$ such that $g\bar g=1$, coincides with the dimension of its Gelfand model.
A group $G(r,p,n)$ turns out to be involutory if and only if $\GCD(p,n)=1,2$ \cite[Theorem 4.5]{Ca2}. 

The model $(M.\varrho)$ provided in \cite{Ca2} looks like this:
\begin{itemize}
 \item $M$ is a formal vector space spanned by all absolute involutions $I(r,p,n)^*$ of the dual group  $G(r,p,n)^*$: $$M\eqdef \bigoplus_{v\in I(r,p,n)^*}\mathbb C C_v;$$
\item $\varrho:G(r,p,n)\rightarrow GL(M)$ works, basically, as an \textit{absolute conjugation }of $G(r,p,n)$ on the elements indexing the basis of $M$:
\begin{equation}\label{basically}
\varrho(g) (C_v) \eqdef \psi(g,v) C_{|g|v|g|^{-1}},
 \end{equation}
$\psi(g,v)$ being a scalar and $|g|$ being the natural projection of $g$ in the symmetric group $S_n$.
\end{itemize}

If $g, h \in G(r,p,n)^*$ we say that $g$ and $h$ are \emph{$S_n$-conjugate} if there
exists $\sigma \in S_n$ such that $g=\sigma h\sigma^{-1}$, and we call \emph{$S_n$-conjugacy classes}  the corresponding equivalence classes.
If $c$ is a $S_n$-conjugacy class of absolute involutions in $I(r,p,n)^*$
we denote by $M(c)$ the subspace of $M$ spanned by the basis elements $C_v$ indexed by the absolute involutions
$v$ belonging to the class $c$. Then it is clear from \eqref{basically} that we have a decomposition
$$
M=\bigoplus_c M(c)\,\,\,\textrm{ as $G(r,p,n)$-modules,}
$$
where the sum runs through all $S_n$-conjugacy classes of absolute
involutions in $I(r,p,n)^*$. It is natural to ask if we can describe the
irreducible decomposition of the submodules $M(c)$, and the main goal of this paper is to answer to this question for every group $G(r,p,n)$ with $\GCD(p,n)=1,2$. The special case of this result for the symmetric group $S_n=G(1,n)$ was established in \cite{IRS}, while the corresponding result for wreath products $G(r,1,n)$ has been recently proved by the present authors in \cite{CaFu}. Though the main result of this paper is a generalization of \cite{CaFu}, we should mention that the proof is not, in the sense that we will actually make use here of the main results of \cite{CaFu}.

The decomposition of the submodules $M(c)$ in this wider setting is much more subtle. Indeed, when $\GCD(p,n)=2$, the Gelfand model $M$  splits also in a different way as the direct sum of two distinguished modules: the symmetric submodule $M^0$, which is spanned by the elements $C_v$ indexed by symmetric absolute involutions, and the antisymmetric submodule $M^1$, which is defined similarly. This decomposition is compatible with the one described above in the sense that every submodule $M(c)$ is contained either in the symmetric  or in the antisymmetric submodule. The existence of the antisymmetric submodule and of the submodules $M(c)$ contained therein  will reflect in a very precise way the existence of split representations for these groups. 
The study of the irreducible decomposition of $M(c)$, when $c$ is made up of antisymmetric elements, requires a particular machinery developed in Sections \ref{DFT}, \ref{Asym} and \ref{anclas} that was not needed in the case of wreath products $G(r,n)$, where the antisymmetric submodule vanishes and so the Gelfand model coincides with its symmetric submodule.

Here is a plan of this paper. In Section \ref{not} we collect the background of preliminary results that are needed to afford the topic. Here an introduction to the groups $G(r,p,q,n)$ can be found, as well as the description of its irreducible representations, and a brief account of the projective Robinson-Schensted correspondence. In Section \ref{model}, for the reader's convenience, we recall the important definition of symmetric and antisymmetric elements given in \cite{Ca2} and the Gelfand model constructed therein. Also a brief account of the main result for the case of $G(r,n)$ can be found here. Section \ref{Dn} consists of an outline of the proof of the main results of this work for the special case of Weyl groups of type $D$. Afterwards, the more general case of all involutory groups of the form $G(r,p,n)$ is treated in full detail. Section \ref{splitclasses} is devoted to the description of the conjugacy classes of such groups. In Section \ref{DFT} we study the discrete Fourier transform, a tool which will be used later in Section \ref{Asym}, where the irreducible decomposition of the antisymmetric submodule  is treated. Section \ref{anclas} then provides an explicit description of the irreducible decomposition of the modules $M(c)$ contained in the antisymmetric submodule. Section \ref{Sym} affords the irreducible decomposition of the submodule $M(c)$, where $c$ is any $S_n$-conjugacy class of symmetric absolute involutions, and Section \ref{last} contains a general result, Theorem \ref{verymain}, that includes all partial results of the previous sections as well as a further slight generalization to all groups $G(r,p,q,n)$ satisfying $\GCD(p,n)=1,2$.

\section[not]{Notation and prerequisites}\label{not}

We let $\mathbb Z $ be the set of integer numbers and $\mathbb N$ be the set of nonnegative integer numbers. For $a,b\in \mathbb Z$, with $a\leq b$ we let $[a,b]=\{a,a+1,\ldots,b\}$ and, for $n\in \mathbb N$ we let $[n]\eqdef[1,n]$. For $r\in\mathbb N$, $r>0$, we let $\mathbb Z_r\eqdef \mathbb Z /r\mathbb Z$. and we denote by $\zeta_r$ the primitive $r$-th root of unity $\zeta_r\eqdef e^{\frac{2\pi i}{r}}$.

The group $G(r,n)$ consists of all the $n\times n$ complex matrices satysfying the following conditions:
\begin{itemize}
\item the non-zero entries are $r$-th roots of unity;
\item there is exactly one non-zero entry in every row and every column.
\end{itemize}
Let now $p|r$. The group $G(r,p,n)$ is the subgroup of $G(r,n)$ of the elements verifying one extra condition:
\begin{itemize}
\item if we write every non-zero element as a power of $\zeta_r$, the sum of all the exponents of $\zeta_r$ appearing in the matrix is a multiple of $p$.
\end{itemize}
We denote by $z_i(g)\in \mathbb Z_r$ the exponent of $\zeta_r$ appearing in the  $i^{\tho}$ row of $g$. We say that $z_i(g)$ is the \emph{color} of $i$ in $g$ and the sum  $z(g)\eqdef z_1(g)+\cdots z_n(g)$ will be called the \emph{color} of $g$. 

It is sometimes convenient to use alternative notation to denote an element in $G(r,n)$, other than the matrix representation. We write $g=[\sigma_1^{z_1},\ldots,\sigma_n^{z_n}]$ meaning that, for all $j\in [n]$, the unique nonzero entry in the $j^{\tho}$ row appears in the $\sigma_j^{\tho}$ column and equals $\zeta_r^{z_j}$ (i.e. $z_j(g)=z_j$). We call this the \emph{window} notation of $g$. We also write in this case $g_i=\sigma_i^{z_i}$. Observe that $[\sigma_1,\ldots,\sigma_n]$ is actually a permutation in $S_n$ and we denote it by $|g|$. Elements of $G(r,n)$  also have a cyclic decomposition which is analogous to the cyclic decomposition of permutations. A \emph{cycle} $c$ of $g\in G(r,n)$ is an object of the form $c=(a_1^{z_{a_1}},\ldots,a_k^{z_{a_k}})$, where $(a_1,\ldots,a_k)$ is a cycle of the permutation $|g|$, and $z_{a_i}=z_{a_i}(g)$ for all $i\in[k]$. We let $k$ be the \emph{length} of $c$, $z(c)\eqdef z_{a_1}+\cdots+z_{a_k}$ be the \emph{color} of $c$, and $\Supp(c) \eqdef \{a_1,\ldots,a_k\}$ be the \emph{support} of $c$. We will sometimes
  write an element $g\in G(r,n)$ as the product of its cycles.
For example if $g\in G(3,6)$ has window notation $g=[3^0,4^1,6^1,2^0,5^2,1^2]$ we have that the cyclic decomposition of $g$ is given by $g=(1^0,3^1,6^2)(2^1,4^0)(5^2)$.
Note that we use square brackets for the window notation and round brackets for the cyclic notation.

If $\nu=(n_0,\ldots,n_k)$ is a composition of $n$ we let $G(r,\nu)\eqdef G(r,n_0)\times\cdots \times G(r,n_k)$ be the (Young) subgroup of $G(r,n)$ given by
$$
G(r,\nu)=\{[\sigma_1^{z_1},\ldots,\sigma_n^{z_n}]\in G(r,n):\,\sigma_i\leq n_0+\cdots +n_j\textrm{ if and only if }i\leq n_0+\cdots+ n_j\}.
$$
If $S\subseteq [n]$ we also let
$$
G(r,S)=\{[\sigma_1^{z_1},\ldots,\sigma_n^{z_n}]\in G(r,n):\,\sigma_i^{z_i}=i^0\textrm{ for all }i\notin S\}.
$$
Consider a partition $\lambda=(\lambda_1,\ldots,\lambda_l)$ of $n$. The \emph{Ferrers diagram of shape $\lambda$} is a collection of boxes, arranged in left-justified rows, with $\lambda_i$ boxes in row $i$.
We denote by $\Fer(r,n)$ the set of $r$-tuples $(\lambda^{(0)},\ldots,\lambda^{(r-1)})$ of Ferrers diagrams such that $\sum |\lambda^{(i)}|=n$.

The set of conjugacy classes of $G(r,n)$ is naturally parametrized by $\Fer(r,n)$ in the following way. If $(\alpha^{(0)},\ldots,\alpha^{(r-1)})\in \Fer(r,n)$ we let $m_{i,j}$ be the number of parts of $\alpha^{(i)}$ equal to $j$. Then the set
$$
\mathrm{cl}_{\alpha^{(0)},\ldots,\alpha^{(r-1)}}=\{g\in G(r,n):g\textrm{ has $m_{i,j}$ cycles of color $i$ and length $j$}\}
$$ 
is a conjugacy class of $G(r,n)$, and all conjugacy classes are of this form.

The set $\Irr(r,n)$ of the irreducible complex representations of $G(r,n)$ is also parametrized by the elements of $\Fer(r,n)$: $\Irr(r,n)=\{\rho_{(\lambda^{(0)},\ldots,\lambda^{(r-1)})},\textrm{ with }(\lambda^{(0)},\ldots,\lambda^{(r-1)})\in \Fer(r,n)\}
$. These representations are described in the following result (where we use the symbol $\otimes$ for the internal tensor product of representations and the symbol $\odot$ for the external tensor product of representations).
\begin{prop}\label{rapp di grn}
Let $\lambda=(\lambda^{(0)},\ldots,\lambda^{(r-1)})\in \Fer(r,n)$, $n_i=|\lambda^{(i)}|$ and $\nu=(n_0,\ldots,n_{r-1})$. The irreducible representation $\rho_{\lambda}$ of $G(r,n)$ is given by
  $$\rho_{\lambda}= \Ind_{G(r,\nu)}^{G(r,n)} \left(\bigodot_{i=0}^{r-1}(\gamma_{n_i}^{\otimes i}\otimes \tilde \rho_{\lambda^{(i)}})\right),$$
where:
\begin{itemize}
        
        \item $\tilde \rho_{\lambda^{(i)}}$ is the natural extension to $G(r,n_i)$ of the irreducible (Specht) representation $\rho_{\lambda^{(i)}}$ of $S_{n_i}$, i.e. $\tilde \rho_{\lambda^{(i)}}(g)\eqdef \rho_{\lambda^{(i)}}(|g|)$ for all $g\in G(r,n_i)$.
        \item $\gamma_{n_i}$ is the
1-dimensional
  representation of $G(r,n_i)$ given by
\begin{align*}
 \gamma_{n_i}:G(r,n_i)&\rightarrow  \mathbb{C}^*\\
 g&\mapsto \zeta_r^{z(g)}.
\end{align*}
      \end{itemize}
\end{prop}
Let us consider a group $G(r,p,n)$. Given $q\in \mathbb{N}$ such that $q|r, pq |rn$, $G(r,p,n)$ contains a unique cyclic scalar subgroup $C_q$ of order $q$. In this case, we can consider the quotient group (see \cite[Section 4]{Ca1})$$G(r,p,q,n)\eqdef\dfrac{G(r,p,n)}{C_q} .$$
A group of this form is called a \textit{projective reflection group}.
Since the conditions of existence of $G(r,p,q,n)$ are symmetric with respect to $p$ and $q$, we can give the following
\begin{defn}
 Let $G=G(r,p,q,n)$ as above. Its \emph{dual group} $G^*$ is the group obtained from $G$ by simply exchanging the roles of $p$ and $q$:
$$G^*\eqdef G(r,q,p,n).$$
\end{defn}
If $\lambda=(\lambda^{(0)},\ldots,\lambda^{(r-1)}) \in \Fer(r,n)$ we define the \emph{color} of $\lambda$ by $z(\lambda)=\sum_i i|\lambda^{(i)}|$ and, if $p|r$ we let $\Fer(r,p,1,n)\eqdef \{\lambda\in \Fer(r,n):z(\lambda)\equiv0 \mod p\}$. The set of irreducible representations $\Irr(r,1,q,n)$ of the group $G(r,1,q,n)=G(r,q,n)^*$ is given by those representations of $G(r,n)$ whose kernel contains the scalar cyclic subgroup $C_q$. It follows from this observation and the description in Proposition \ref{rapp di grn} that
$$
\Irr(r,1,q,n)=\{\rho_{\lambda}:\, \lambda\in \Fer(r,q,1,n)\}.
$$

The irreducible representations of $G(r,p,q,n)$ can now be deduced essentially by Clifford theory from those of $G(r,1,q,n)$. We apply this theory in this case as the final description will be very explicit. 

Once restricted to $G(r,p,q,n)$, the irreducible representations of $G(r,1,q,n)$ may not be irreducible anymore. Let us see which of them split into more than one $G(r,p,q,n)$-module. Consider the \emph{color representation} $\gamma_n:G(r,n)\rightarrow  \mathbb{C}^*$ given by $g\mapsto \zeta_r^{z(g)}$.
We note that $\gamma_n^{r/p}$ is a well-defined representation of $G(r,1,q,n)$ of order $p$ and that the kernel of the cyclic group $\Gamma_p=\langle\gamma_n^{r/p}\rangle$ of representations of $G(r,1,q,n)$ is $G(r,p,q,n)$. The group $\Gamma_p$ acts on the set of the irreducible representations of $G(r,1,q,n)$ by internal tensor product. If we let $n_i=|\lambda^{(i)}|$ and $\nu=(n_0,\ldots,n_{r-1})$ we have that this action is given by
\begin{eqnarray}\label{azione}
 \gamma_n^{r/p} \otimes  \rho_{\lambda^{(0)},\ldots,\lambda^{(r-1)}}&=& \Ind_{G(r,\nu)}^{G(r,n)}\big(\gamma_n^{r^p}|_{G(r,\nu)}\big)\otimes
\bigodot_{i=0}^{r-1}(\gamma_{n_i}^{\otimes i}\otimes \tilde \rho_{\lambda^{(i)}})\\
&=&\rho_{\lambda^{(r-{r/p})},\ldots, \lambda^{(r-1)}, \lambda^{(0)}, \ldots,\lambda^{(r-1-r/p)}},\nonumber
\end{eqnarray}
and so it simply corresponds to a shift of $r/p$ of the indexing partitions.

It is now natural to let $\Fer(r,q,p,n)$ be the set of orbits in $\Fer(r,q,1,n)$ with respect to the action of $\Gamma_p$ described in \eqref{azione} and if $\lambda=(\lambda^{(0)},\ldots,\lambda^{(r-1)})\in \Fer(r,q,1,n)$ we denote by $[\lambda]$ or $[\lambda^{(0)},\ldots,\lambda^{(r-1)}]\in \Fer(r,q,p,n)$ the corresponding orbit.  Moreover, if $[\lambda^{(0)},\ldots,\lambda^{(r-1)}]\in \Fer(r,q,p,n)$ we let $\St_{[\lambda^{(0)},\ldots,\lambda^{(r-1)}]}$ be the set of \emph{standard multitableaux} obtained by filling the boxes of any element in $[\lambda^{(0)},\ldots,\lambda^{(r-1)}]$ with all the numbers from $1$ to $n$ appearing once, in such way that rows are increasing from left to right and columns are increasing from top to bottom (see \cite[Section 6]{Ca1}).

We will now state a theorem which applies in full generality to every group $G(r,p,q,n)$, and fully clarifies the nature of its irreducible representations.  Here and in what follows, if $\lambda\in \Fer(r,n)$ we let $m_p(\lambda)=|\Stab_{\Gamma_p}(\lambda)|$ and we observe that if $[\lambda]=[\mu]\in \Fer(r,q,p,n)$ then $m_p(\lambda)=m_p(\mu)$.
\begin{thm}\label{irrepgrpqn}
   The set of irreducible representations $\Irr(r,p,q,n)$ of $G(r,p,q,n)$ can be parametrized in the following way 
\begin{eqnarray*}
 \Irr(r,p,q,n)&=&\{\rho_{[\lambda]}^j:[\lambda]\in \Fer(r,q,p,n)\textrm{ and } j\in [0,m_p(\lambda)-1]\},
 \end{eqnarray*}
so that the following conditions are satisfied:
\begin{itemize}
\item   $\dim(\rho_{[\lambda]}^j)=|\St_{[\lambda]}|$ for all $[\lambda]\in \Fer(r,q,p,n)$ and $j\in [0,m_p(\lambda)-1]$;
\item $
\mathrm{Res}^{G(r,1,q,n)}_{G(r,p,q,n)}(\rho_{\lambda})\cong \bigoplus_{j} \rho_{[\lambda]}^{j}$ for all $\lambda\in \Fer(r,q,1,n)$.
\end{itemize}

\end{thm}

If $m_p(\lambda)=1$ we sometimes write $\rho^{}_{[\lambda]}$ instead of $\rho^0_{[\lambda]}$ and we say that this is an \emph{unsplit} representation. On the other hand, whenever $m_p(\lambda)>1$, we say that all representations of the form $\rho^j_{[\lambda]}$ are \emph{split} representations. We will come back to this description of the irreducible representations in Section \ref{DFT} with more details.

Let us now turn to give a brief account of the \textit{projective Robinson-Schensted correspondence}, which is an extension of the Robinson-Schensted correspondence for the symmetric group to all groups of the form $G(r,p,q,n)$. This is a surjective map
 $$G(r,p,q,n)\longrightarrow \bigcup_{[\lambda]\in \Fer(r,p,q,n)}\St_{[\lambda]} \times \St_{[\lambda]}$$
such that, if $P,Q\in\St_{[\lambda]}$ then  the cardinality of the inverse image of $(P,Q)$ is equal to $m_q(\lambda)$. In particular we have that this correspondence is a bijection if and only if $\GCD(q,n)=1$. We refer the reader to \cite[Section 10]{Ca1} for the precise definition and further properties of this correspondence.\\
Note that while in Theorem \ref{irrepgrpqn} we use elements $[\lambda]\in \Fer(r,q,p,n)$, in the projective Robinson-Schensted correspondence the elements $[\lambda]$ involved belong to $\Fer(r,p,q,n)$. This is one of the reasons why it is natural to look at the dual groups when studying the representations theory of any projective reflection group of the form $\qc$. 

\section{The model and its natural decomposition}\label{model}

A Gelfand model for a group $G$ is a $G$-module affording each irreducible representation of $G$ exactly once. A Gelfand model was constructed in \cite{Ca2} for every $G(r,p,q,n)$ such that $\GCD(p,n)=1,2$. In order to illustrate it, we need to introduce some new concepts and definitions.

An element $g\in G(r,p,q,n)$ is an \emph{absolute involution} if $g\bar{g}=1$, $\bar{g}$ being the complex conjugate of $g$ (note that this is well-defined since complex conjugation stabilizes the cyclic scalar group $C_q$).
We denote by $I(r,n)$ the set of the absolute involutions of the group  $G(r,n)$ and we similarly define $I(r,p,n)$ and $I(r,p,q,n)$. Moreover we let $I(r,p,n)^*$ stand for the set of the absolute involutions of the group $G(r,p,n)^*$.

The absolute involutions in $I(r,p,q,n)$ can be either symmetric or antisymmetric, according to the following definition:
\begin{defn}
 Let $v\in G(r,p,q,n)$. We say that it is:
\begin{itemize}
 \item symmetric, if every lift of $v$ in $G(r,n)$ is a symmetric matrix;
\item antisymmetric, if every lift of $v$ in $G(r,n)$ is an antisymmetric matrix.
\end{itemize}
\end{defn}
We observe that while a symmetric element is always an absolute involution, an antisymmetric element of $G(r,p,q,n)$ is an absolute involution if and only if $q$ is even (see \cite[Lemma 4.2]{Ca2}). 
Antisymmetric elements can also be characterized in terms of the projective Robinson-Schensted correspondence (see \cite[Lemma 4.3]{Ca2}:
\begin{lem}\label{anrs}
Let $v\in G(r,n)$. Then the following are equivalent
\begin{enumerate}
\item $v$ is antisymmetric;
\item $r$ is even and $v\mapsto ((P_0,\ldots,P_{r-1}),(P_{\frac{r}{2}},\ldots,P_{r-1},P_0,\ldots,P_{\frac{r}{2}}))$ for some $(P_0,\ldots,P_{r-1})\in \St_{\lambda}$ and $\lambda\in \Fer(r,n)$ by the projective Robinson-Schensted correspondence.
\end{enumerate}
\end{lem}
Now we can deduce the following combinatorial interpretation for the number of antisymmetric elements in a projective reflection group. Since we often deal with even integers, here and in the rest of this paper we let $k'\eqdef \frac{k}{2}$, whenever $k$ is an even integer. \\
\begin{prop}\label{casym} Let $\mathrm{asym}(r,q,p,n)$ be the number of antisymmetric elements in $G(r,q,p,n)$. Then
$$
\mathrm{asym}(r,q,p,n)=\sum_{[\mu,\mu]\in \Fer(r,q,p,n)}\St_{[\mu,\mu]},
$$
where $[\mu,\mu]\in \Fer(r,q,p,n)$ means that $[\mu,\mu]$ varies among all elements in $\Fer(r,q,p,n)$  of the form \\$[\mu^{(0)},\ldots,\mu^{(r'-1)},\mu^{(0)},\ldots,\mu^{(r'-1)}]$, for some $\mu=(\mu^{(0)},\ldots,\mu^{(r'-1)})\in\Fer(r',n')$.
\end{prop}
\begin{proof}
   
Observe that if $v\in G(r,q,n)$ is antisymmetric and $(P_0,\ldots,P_{r-1})$ and $\lambda$ are as in Lemma \ref{anrs}, then necessarily $\lambda\in \Fer(r,q,1,n)$ is of the form $\lambda=(\mu,\mu)$, for some $\mu\in \Fer(r',n')$. So, if $v\mapsto(P,Q)$ is antisymmetric we have that $P$ is any element in $\St_{(\mu,\mu)}$ for some $\mu\in \Fer(r',n')$ whilst $Q$ is uniquely determined by $P$. 
So we deduce that
$$
\mathrm{asym}(r,q,1,n)=\sum_{(\mu,\mu)\in \Fer(r,q,1,n)}\St_{(\mu,\mu)}.
$$

The result now follows since every antisymmetric element in $G(r,q,p,n)$ has $p$ distinct lifts in $G(r,q,n)$ and any element in $\St_{[\mu,\mu]}$ has $p$ distinct lifts in $\cup_{(\nu,\nu)\in [\mu,\mu]} \St_{(\nu,\nu)}$.

 \end{proof}
Before describing the Gelfand model for the involutory reflection groups we need to recall some further notation from \cite{Ca2}.
If $\sigma, \tau \in S_n$ with $\tau^2=1$ we let 
\begin{itemize}
\item$\Inv(\sigma)=\{\{i,j\}:(j-i)(\sigma(j)-\sigma(i))<0\}$;
\item$ \Pair(\tau)=\{\{i,j\}:\tau(i)=j\neq i\}$;
\item $\inv_{\tau}(\sigma)=|\{\Inv(\sigma)\cap \Pair(\tau)|$.
\end{itemize}

If $g\in G(r,p,q,n)$, $v\in I(r,q,p,n)$, $\tilde g$ any lift of $g$ in $G(r,p,n)$ and $\tilde v$ any lift of $v$ in $G(r,q,n)$, we let 
\begin{itemize}
\item $\inv_v(g)=\inv_{|v|}(|g|);$
\item $\langle g,v\rangle=\sum_{i=1}^n z_i(\tilde g)z_i(\tilde v) \in \mathbb Z_r$;
\item $a(g,v)=z_1(\tilde v)-z_{|g|^{-1}(1)}(\tilde v)\in \mathbb Z_{r}.$
\end{itemize}
The verification that $\langle g,v\rangle$ and $a(g,v)$ are well-defined is straightforward.

We are now ready to present the Gelfand model constructed in \cite{Ca2}.

\begin{thm}\label{modello}Let $\GCD(p,n)=1,2$ and let $$M(r,q,p,n)\eqdef \bigoplus_{v\in I(r,q,p,n)}\mathbb C C_v$$ and $\varrho:G(r,p,q,n)\rightarrow GL(M(r,q,p,n))$ be defined by
\begin{equation} \label{action}
\varrho(g) (C_v) \eqdef \left\{\begin{array}{ll}
\zeta_r^{\langle g,v\rangle} (-1)^{\inv_v(g)}C_{|g|v|g|^{-1}} & \textrm{ if $v$ is symmetric}\\
\zeta_r^{\langle g,v\rangle}\zeta_r^{a(g,v)}C_{|g|v|g|^{-1}}& \textrm{ if $v$ is antisymmetric}.\end{array}\right.
 \end{equation}
Then $(M(r,q,p,n),\varrho)$ is a Gelfand model for $G(r,p,q,n)$.
\end{thm}

Let us have a short digression to recall what happens for the wreath products $G(r,n)$. In this case the setting is much simpler since the group $G(r,n)$ coincides with its dual, there are no split representations, and no antisymmetric absolute involutions. Moreover the absolute involutions are characterized as those elements $v$ satisfying
$$
v\mapsto(P,P)
$$
for some $P\in \St_\lambda$, $\lambda\in \Fer(r,n)$, via the Robinson-Schensted correspondence. We write in this case $\Sh(v)\eqdef\lambda$.
If $c$ is an $S_n$-conjugacy class of absolute involutions in $G(r,n)$ we also let $\Sh(c)=\cup_{v\in c}\Sh(v)\subset \Fer(r,n)$. The main result in \cite{CaFu} is the following theorem of compatibility with respect to the Robinson-Schensted correspondence of the irreducible decomposition of the submodules $M(c)$ of $M(r,1,1,n)$ defined in the introduction.
\begin{thm}\label{wreath}
Let $c$ be a $S_n$-conjugacy class of absolute involutions in $G(r,n)$. Then
the following decomposition holds:
$$M(c) \cong \bigoplus_{\begin{subarray}{c}
\lambda\in \Sh(c)\\
 \end{subarray}}
  \rho_\lambda.$$
\end{thm}
The main target of this paper is to establish a result analogous to Theorem \ref{wreath} for all groups $G(r,p,q,n)$ which are involutory (Theorem \ref{verymain}). In this general context we also have the decomposition
$$
M=M^0\oplus M^1,
$$
where $M^0$ is the {symmetric submodule} of $M$, i.e. the submodule spanned by all elements $C_v$ indexed by symmetric absolute involutions, and $M^1$ is the {antisymmetric submodule} defined similarly. In fact, the main step in the description of the irreducible decomposition of the modules $M(c)$, will be the intermediate result that provides the irreducible decomposition of the symmetric and the antisymmetric submodules.

\section{An outline: the irreducible decomposition of $M(c)$ in type D}\label{Dn}

In this section we give an outline of the proofs of the main results in the special case $D_{n}=G(2,2,n)$. Here we may take advantage of some results which are already known in the literature, such as the description of the split conjugacy classes and of the split representations and their characters. 

 The irreducible representations of $B_{n}$ are indexed by elements $(\lambda,\mu)\in \Fer(2,n)$. If $(\lambda,\mu)\in \Fer(2,n)$ is such that $\lambda\neq \mu$ then the two representations $\rho_{(\lambda,\mu)}$ and $\rho_{(\mu,\lambda)}$, when restricted to $D_n$ are irreducible and isomorphic by Theorem \ref{irrepgrpqn}, and we denote this representation by $\rho_{[\lambda,\mu]}$.
 If $n=2m$ is even, Theorem \ref{irrepgrpqn} also implies that the irreducible representations of $B_{2m}$ of the form $\rho_{(\mu,\mu)}$, when restricted to $D_n$, split into two irreducible representations that we denote by $\rho^0_{[\mu,\mu]}$ and $\rho^1_{[\mu,\mu]}$.\\
 The conjugacy classes of $B_n$ contained in $D_n$ are those indexed by ordered  pairs of partitions $(\alpha,\beta)$, with $|\beta|\equiv 0 \mod 2$. They all do not split as $D_n$-conjugacy classes with the exception of those indexed by $(2\alpha, \emptyset)$, which split into two $D_n$-conjugacy classes that we denote by  $\mathrm{cl}_{2\alpha}^0$ and $\mathrm{cl}_{2\alpha}^1$ (and we make the convention that $\mathrm{cl}_{2\alpha}^0$ is the class containing all the elements $g$ belonging to the $B_n$-class labelled by $(2\alpha,\emptyset)$ and satisfying $z_i(g)=0$ for all $i\in[n]$).

 The characters of the unsplit representations are clearly the same as those of the corresponding representations of the groups $B_n$ (being the corresponding restrictions). The characters of the split representations are given by the following result (see \cite{Stembridge, Pf}).
\begin{lem}\label{splitchar}
Let $g\in D_{2m}$, and $\mu\vdash m$. Then
$$\chi_{[\mu,\mu]}^\epsilon(g)=\left\{\begin{array}{ll}\frac{1}{2}\chi_{(\mu,\mu)}(2\alpha, \emptyset)+(-1)^{\epsilon +\eta }2^{\ell(\alpha)-1}\chi_{\mu}(\alpha),&\textrm{ if }g\in \mathrm{cl}_{2\alpha}^\eta;\\ \frac{1}{2} \chi_{(\mu,\mu)}(g),&\textrm{ otherwise.} \end{array}\right.$$
where $\epsilon,\eta=0,1 $, $\chi_{(\mu,\mu)}$ is the character of the $B_{2m}$-representation $\rho_{(\mu,\mu)}$ and $\chi_{\mu}$ is the character of $S_{m}$ indexed by $\mu$.
\end{lem}
In Section \ref{splitclasses} we prove a generalization of this result valid for all groups $G(r,p,n)$ such that $\GCD(p,n)=2$.

An antisymmetric element in $B_{2m}$ is necessarily the product of cycles of length 2 and color 1, i.e. cycles of the form $(a^0,b^1)$. It follows that the antisymmetric elements of $B_{2m}$, and hence also those of $B_{2m}/\pm I$, are all $S_n$-conjugate. This is a special feature of this case and is not true for generic involutory reflection groups (see Section \ref{anclas}). We denote by $c^1$ the unique $S_n$-conjugacy class of antisymmetric absolute involutions in $B_{2m}/\pm I$ and we will now find out which of the irreducible representations of $D_{2m}$ are afforded
by the antisymmetric submodule $M^1$, which coincides in this case with $M(c^1)$.
The crucial observation  is the following result, which is a straightforward consequence of Lemma \ref{splitchar}.
\begin{remark}\label{lambda} Let $g\in D_{2m}$. Then

\begin{equation}
\sum_{\mu\vdash m}(\chi_{[\mu,\mu]}^0-\chi_{[\mu,\mu]}^1)(g)=\left\{ \begin{array}{ll}
                                                                                 (-1)^\eta 2^{\ell(\alpha)}\sum_{\mu\vdash m}\chi_{\mu}(\alpha),&  \mbox{if } g\in \mathrm{cl}_{2\alpha}^\eta;\\
0,& \textrm{otherwise.}
                                                                                 \end{array}\right.
\end{equation}

\end{remark}

The main result here is the following.
\begin{thm}\label{splitD2m}
Let $c^1$ be the $S_n$-conjugacy class consisting of the antisymmetric involutions in $D_n^*=B_n/\pm I$. Then
 $$M(c^1)\cong \bigoplus_{\mu\vdash m}\rho_{[\mu,\mu]}^1.$$
\end{thm}

\begin{proof}We present here a sketch of the proof only since this result will be generalized and proved in full detail in Section \ref{Asym}.
   
Consider the two representations $\phi^0$ and $\phi^1$ of $D_{2m}$ on the vector space $M(c^1)$ given by 
$$\phi^0(g)(C_v)\eqdef(-1)^{\langle g,v\rangle}C_{|g|v|g|^{-1}},  \qquad \phi^1(g)(C_v)\eqdef(-1)^{\langle g,v\rangle}(-1)^{a(g,v)}C_{|g|v|g|^{-1}}$$
(notice that $\phi^1(g)=\varrho(g)|_{M(c^1)}$).
We will simultaneously prove that $$\chi_{\phi^0}=\sum_{\mu\vdash m}\chi^0_{[\mu,\mu]} \quad \mbox{ and }\quad \chi_{\phi^1}=\sum_{\mu\vdash m}
\chi^1_{[\mu,\mu]},$$
the latter equality being equivalent to the statement that we have to prove.
To this end, we observe that it will be enough to show that
\begin{equation}\label{agl}
  \chi_{\phi^0}-\chi_{\phi^1}= \sum_{\mu\vdash m}(\chi_{[\mu,\mu]}^0-\chi_{[\mu,\mu]}^1).
\end{equation}
In fact if \eqref{agl} is satisfied, we have

\begin{equation}
\chi_{\phi^1}+\sum_{\mu\vdash m}\chi^0_{[\mu,\mu]}=
\chi_{\phi^0}+\sum_{\mu\vdash m}\chi^1_{[\mu,\mu]}.
\end{equation}
Now, since the irreducible characters are linearly independent, it follows that $\phi^0$ has a subrepresentation isomorphic to $\oplus\rho_{[\mu,\mu]}^0$, and similarly for $\phi^1$. By Theorem \ref{irrepgrpqn} and Proposition \ref{casym} we also have that

\begin{equation*}
\sum_{\mu\vdash m}\dim(\rho^0_{[\mu,\mu]})=\sum_{[\mu,\mu]\in \Fer(2,1,2,2m)}|\St_{[\mu, \mu]}|=|c^1| =\dim (\phi^0),
\end{equation*}
and, analogously, $\sum_{\mu\vdash m}\dim(\rho^1_{[\mu,\mu]})=\dim (\phi^1)$, and we are done.

To prove Equation \eqref{agl}, one has to compute explicitly the difference  $\chi_{\phi^0}-\chi_{\phi^1}$ and show that this agrees with the right-hand side of Equation \eqref{lambda}. To this end we will need to observe that $\sum {\chi_{\mu}}$ is actually the character of a Gelfand model of the symmetric group $S_m$, which has an already known combinatorial interpretation (see, e.g., \cite[Proposition 3.6]{Ca2}).
\end{proof}

Let us now consider a class $c$ of symmetric involutions in $D_n^*=B_n/\pm I$ (note that in this case an absolute involution is actually an involution since all the involved matrices are real). The lift of $c$ to $B_n$ is the union of two $S_n$-conjugacy classes $c_1$ and $c_2$ of $B_n$ that may eventually coincide.  Since we already know the irreducible decompositions of $M(c_1)$ and of $M(c_2)$ as  $B_n$-modules (by Theorem \ref{wreath}) and hence also as  $D_n$-modules (by Theorem \ref{irrepgrpqn}), the main point in the proof of the following result will be that $M(c)$ is actually isomorphic to a subrepresentation of the restriction of $M(c_1)\oplus M(c_2)$ to $D_n$,  together with straightforward applications of Theorems \ref{modello} and \ref{splitD2m}.

\begin{thm}Let $c$ be a $S_n$-conjugacy class of symmetric involutions in $D_n^*=B_n/\pm I$. Then
$$
M(c)\cong \bigoplus_{[\lambda,\mu]\in \Sh(c)}\rho^0_{[\lambda,\mu]}.
$$
\end{thm}
In Section \ref{Sym} one can find an explicit simple combinatorial description of the sets $\Sh(c)$, for any symmetric $S_n$-conjugacy class $c$ of absolute involutions in $G(r,p,n)^*$. This is illustrated in the following example. 
\begin{exa}
Let $v\in B_6/\pm I$ be given by $v=[6^1,4^0,3^0,2^0,5^1,1^1]$. Then the $S_n$-conjugacy class $c$ of $v$, has 90 elements and the decomposition of the $D_n$-module $M(c)$ is given by all representations $\rho^0_{[\lambda,\mu]}$, $[\lambda,\mu]\in \Fer(2,1,2,6)$, where both $\lambda$ and $\mu$ are partitions of 3 and have exactly one column of odd length. Therefore

\setlength{\unitlength}{0.7pt}

$$
M(c)\cong \rho_{\left[\begin{picture}(24,10)(-1,-1)\put(0,8){\line(1,0){10}}\put(0,3){\line(1,0){10}}\put(0,-2){\line(1,0){5}}\put(0,-2){\line(0,1){10}}\put(5,-2){\line(0,1){10}}\put(10,3){\line(0,1){5}}\put(12,0){,}\put(17,8){\line(1,0){5}}\put(17,3){\line(1,0){5}}\put(17,-2){\line(1,0){5}}\put(17,-7){\line(1,0){5}}\put(17,-7){\line(0,1){15}}\put(22,-7){\line(0,1){15}}\end{picture}\right]}
\oplus
\rho^0_{\left[\begin{picture}(28,10)(-1,-1)\put(0,8){\line(1,0){10}}\put(0,3){\line(1,0){10}}\put(0,-2){\line(1,0){5}}\put(0,-2){\line(0,1){10}} \put(5,-2){\line(0,1){10}}\put(10,3){\line(0,1){5}}\put(12,0){,}\put(17,8){\line(1,0){10}}\put(17,3){\line(1,0){10}}\put(17,-2){\line(1,0){5}}\put(17,-2){\line(0,1){10}}\put(22,-2){\line(0,1){10}}\put(27,3){\line(0,1){5}}\end{picture}\right]} 
\oplus
\rho^0_{\left[\,\begin{picture}(19,10)(0,-1)\put(0,8){\line(1,0){5}}\put(0,3){\line(1,0){5}}\put(0,-2){\line(1,0){5}}\put(0,-7){\line(1,0){5}} \put(0,-7){\line(0,1){15}} \put(5,-7){\line(0,1){15}}\put(7,0){,}\put(12,8){\line(1,0){5}}\put(12,3){\line(1,0){5}}\put(12,-2){\line(1,0){5}}\put(12,-7){\line(1,0){5}}\put(12,-7){\line(0,1){15}}\put(17,-7){\line(0,1){15}}\end{picture}\right]}.
$$
Note in particular that in this case we obtain one unsplit representation and two split representations.
\end{exa}

\section{On the split conjugacy classes}\label{splitclasses}

In the more general case of any involutory reflection group $G(r,p,n)$, we have not been able to find the nature of the conjugacy classes that split from $G(r,n)$ to $G(r,p,n)$  in the literature. This is the content of the present section.

Let $r$ be even so that it makes sense to talk about even and odd elements in $\mathbb Z_r$. Let $c$ be a cycle in $G(r,n)$ of even length and even color. If $c=(i_1^{z_{i_1}},i_2^{z_{i_2}},\ldots,i_{2d}^{z_{i_{2d}}})$ we define the \emph{signature} of $c$ to be
$$
\sign(c)={z_{i_1}+z_{i_3}+\cdots+z_{i_{2d-1}}}={z_{i_2}+z_{i_4}+\cdots+z_{i_{2d}}}\in \mathbb Z_2,
$$
so that the signature can be either 0 or 1. If $g$ is a product of disjoint cycles of even length and even color we define the \emph{signature} $\sign(g)$ of $g$ as the sum of the signatures of its cycles.
\begin{lem}\label{signature}
 Let $r$ be even and let $c$ be a cycle in $G(r,n)$ of even length and even color. Let $h\in G(r,n)$. Then
$$
\sign(h^{-1}ch)=\sign(c)+{\sum_{j\in |h|^{-1}(\Supp(c))}z_j(h)}\in \mathbb Z_2.
$$
In particular, if $g\in G(r,n)$ is a product of cycles of even length and even color then
$$
\sign(h^{-1}gh)=\sign(g)+{z(h)}\in \mathbb Z_2.
$$
\end{lem}
\begin{proof}
Let $|c|=(i_1,i_2,\ldots,i_{2d})$. If we let $\tau=|h|$ we have that $h^{-1}ch$ is a cycle and $|h^{-1}ch|=(\tau^{-1}(i_1),\ldots,\tau^{-1}(i_{2d}))$. Therefore
\begin{align*}
\sign(h^{-1}ch)&={\sum_{j\textrm{ odd}}
z_{\tau^{-1}(i_j)}(h^{-1}ch)}\\
&={\sum_{j\textrm{ odd}}
z_{\tau^{-1}(i_j)}(h)+z_{i_j}(c)-z_{\tau^{-1}(i_{j+1})}(h)}\\
&=\sign(c)+{\sum_{j\in |h|^{-1}(\Supp(c))}z_j(h)},
\end{align*}
where the sums in the first two lines are meant to be over all odd integers $j\in[2d].$
\end{proof}
It follows from Lemma \ref{signature} that the conjugacy classes $\mathrm{cl}_{\alpha}$ of $G(r,n)$ contained in $G(r,p,n)$, where $\alpha$ has the special form $\alpha=(2\alpha^{(0)},\emptyset,2\alpha^{(2)},\emptyset,\ldots,2\alpha^{(r-2)},\emptyset)$, split in $G(r,p,n)$ into (at least) two conjugacy classes, according to the signature. How about the $G(r,n)$-conjugacy classes of a different form? Do they split as $G(r,p,n)$-classes?

If $G$ is a group and $g\in G$ we denote by ${\rm cl}_G(g)$ the conjugacy class of $g$ and by $C_G(g)$ the centralizer of $g$ in $G$.
If $g\in G(r,p,n)$ then the $G(r,n)$-conjugacy class ${\rm cl}_{G(r,n)}(g)$ of $g$ splits into more than one $G(r,p,n)$-conjugacy class if and only if
$$
 \frac{|{\rm cl}_{G(r,n)}(g)|}{|{\rm cl}_{G(r,p,n)}(g)|}=\frac{\frac{|G(r,n)|}{|C_{G(r,n)}(g)|}}{\frac{|G(r,p,n)|}{|C_{G(r,p,n)}(g)|}}=
 \frac{[G(r,n):G(r,p,n)]}{[C_{G(r,n)}(g):C_{G(r,p,n)}(g)]}=\frac{p}{[C_{G(r,n)}(g):C_{G(r,p,n)}(g)]}>1
$$
 i.e., $${\rm cl}_{G(r,n)}(g)\mbox{ splits if and only if }[C_{G(r,n)}(g):C_{G(r,p,n)}(g)]<p.$$
The following proposition clarifies which conjugacy classes of $G(r,n)$ split in $G(r,p,n)$.

\begin{prop}
 Let $g \in G(r,p,n)$ and let ${\rm cl}(g)$ be its conjugacy class in the group $G(r,n)$. Then the following holds:
\begin{enumerate}
 \item if $\GCD(p,n)=1$, ${\rm cl}(g)$ does not split as a class of $G(r,p,n)$;
\item if $\GCD(p,n)=2$, ${\rm cl}(g)$ splits up into two different classes  of $G(r,p,n)$ if and only if all the cycles of $g$ have: 
\begin{itemize}                                                                                                                                                 \item even length;
\item even color,                                                                                                                                \end{itemize}
\end{enumerate}
i.e., if ${\rm cl}(g)=\mathrm{cl}_{(2\alpha^{(0)},\emptyset,2\alpha^{(2)},\emptyset,\ldots,2\alpha^{(r-2)},\emptyset)}$.

\end{prop}

\begin{proof}Let $G=G(r,n)$ and $H=G(r,p,n)$.
We first make a general observation. If $C_{G}(g)$ contains an element $x$ such that $z(x)\equiv 1 \mod p$, we can split the group $C_{G}(g)$ into cosets modulo the subgroup $\langle x\rangle$: in each coset there is exactly 1 element having color $0 \mod p$ every $p$ elements. Thus,  $$[C_G(g):C_H(g)]=p$$ and ${\rm cl}_G(g)$ does not split in $H$.

 Now let $\GCD(p,n)=1$. Thanks to B\'{e}zout identity, there exist $a,b$ such that $an+bp=1$, i.e. there exists $a$ such that the scalar matrix $\zeta_r^a\Id$ has color $1 \mod p$, so that ${\rm cl}_G(g)$ does not split thanks to the observation above.

Assume now that $\GCD(p,n)=2$. Arguing as above, there exist $a$, $b$ such that $ap+bn= 2$, so we know that $C_G(g)$ contains at least an element $\zeta_r^a\Id$ with color $2 \mod p$.

If there exists at least an element $x$ of odd color in $C_G(g)$, the matrix $(\zeta_r^a\Id)^i\cdot x$ has color $1$ for some $i$, so again ${\rm cl}_G(g)$ does not split in $H$.

On the other hand, if there are no elements of odd color in $C_G(g)$, every coset of $\langle \zeta_r^a\Id\rangle$ have exactly 1 element belonging to $G(r,p,n)$ out of $p'$ elements. 
Thus,  $$[C_G(g):C_H(g)]=p'$$ and ${\rm cl}(g)$ splits into $p/p'=2$ classes.

Let us see when this happens according the the cyclic structure of $g$.
     \begin{enumerate}

    \item  If $g$ has at least a cycle of odd color, say $c$, $c$ is in $C_G(g)$ and ${\rm cl}(g)$ does not split.

    \item  If $g$ has a cycle of odd length, say $c$, then $\zeta_r|c|$ has odd color and is in $C_G(g)$, so ${\rm cl}(g)$ does not split.

   \item  We are left to study the case of $g$ being a product of cycles all having even length and even color. Thanks to Lemma \ref{signature}, every element in $C_{G}(g)$ has even color, so by the above argument ${\rm cl}(g)$ splits into exactly two classes, and we are done.
     \end{enumerate}
\end{proof}

If $2\alpha=(2\alpha^{(0)},\emptyset,2\alpha^{(2)},\emptyset,\ldots,2\alpha^{(r-2)},\emptyset)$ is such that $\mathrm{cl}_{2\alpha}\subset G(r,p,n)$ (i.e. if $\sum 2i{\ell(\alpha_i)}\equiv 0 \mod p$) we denote by $\mathrm{cl}_{2\alpha}^0$ the $G(r,p,n)$-conjugacy class consisting of all elements in $\mathrm{cl}_{2\alpha}$ having signature $0$, and we similarly define $\mathrm{cl}_{2\alpha}^1$.

\section{The discrete Fourier transform}\label{DFT}

Recall from equation \eqref{azione} that there is an action of the cyclic group $\Gamma_p$ generated by $\gamma_n^{r/p}$ on the set of irreducible representations of $G(r,n)$. This action gives us the opportunity of introducing the concept of discrete Fourier transform, which will be essential in what follows about the case $\GCD(p,n)=2$. We will parallel and generalize in this section an argument due to Stembridge (\cite{Stembridge}, Sections 6 and 7B).

We recall, according to \cite{Stembridge}, the following

\begin{defn}
Let $\lambda\in \Fer(r,n)$, $(V, \rho_\lambda)$ be a concrete realization of the irreducible $G(r,n)$-representation $\rho_\lambda$ on the vector space $V$, and  $\gamma$ be a generator for $\stab_{\Gamma_p}(\rho_\lambda)$. An \emph{associator} for the pair $(V,\gamma)$  is an element $S\in GL(V)$ exhibiting an explicit isomorphism of $G(r,n)$-modules between
$$(V, \rho_{\lambda}) \quad \mbox{ and } \quad (V, \gamma \otimes \rho_{\lambda}).$$
By Schur's lemma $S^{m_p(\lambda)}$ is a scalar, and therefore $S$ can be normalized in such a way that $S^{m_p(\lambda)}=1$.
\end{defn}

Recall form  Theorem \ref{irrepgrpqn} that a representation $\rho_\lambda$ of $G(r,n)$ splits  into exactly $m_p(\lambda)$ irreducible representations of $G(r,p,n)$.
\begin{defn}
Let $\lambda\in \Fer(r,n)$ and let $S$ be an associator for the $G(r,n)$-module $(V, \rho_\lambda)$. Then the \emph{discrete Fourier transform} with respect to $S$ is the family of $G(r,p,n)$-class functions $\Delta^i_{\lambda}:G(r,p,n)\rightarrow \mathbb C^*$ given by
\begin{equation*}
 \Delta^i_{\lambda}(h):=\tr(S^i(h)|_V), i \in [0, m_p(\lambda)-1].
\end{equation*}
\end{defn}

A more accurate analysis of how the associator works, shows that the irreducible representations $\rho_{[\lambda]}^i$ are exactly the eigenspaces of the associator $S$, and we make the convention that, once an associator $S$ has been fixed, the representation $\rho_{[\lambda]}^i$ is the one afforded by the eigenspace of $S$ of eigenvalue $\zeta_{m_p(\lambda)}^i$. Therefore 
\begin{equation}\label{eccocosacivuole}
\Delta^i_\lambda(h)= \sum_{j=0}^{m_p(\lambda)-1}\zeta_{m_p(\lambda)}^{ij}\chi_{[\lambda]}^j(h),
\end{equation}
for all $h\in G(r,p,n)$, $\chi_{[\lambda]}^j$ being the character of the split representation $\rho_{[\lambda]}^j$ of $G(r,p,n)$.

Now let us consider a representation $\rho_{\lambda}$. Looking at the action described in \eqref{azione}, we see that $m_p(\lambda)=|\Stab_{\Gamma_p}(\rho_{\lambda})|=s$ only if $\lambda=(\lambda^{(0)},\ldots, \lambda^{(r-1)})$ consists of a smaller pattern repeated $s$ times. It follows that $m_p(\lambda)$ is necessarily a divisor of both $n$ and $p$. 

In particular, if $\GCD(p,n)=2$, $m_p(\lambda)=1,2$ and so the stabilizer of a representation with respect to $\Gamma_p$ can either be trivial or be $\{\Id,
\gamma_n^{r'}\}$. 
\begin{notaz}
From now on, when $\GCD(p,n)=2$ we use for the
representation $\gamma_n^{r'}(g)=(-1)^{z(g)}$ the
notation $\delta(g)$.
\end{notaz}
 When $\Stab_{\Gamma_p}(V)=\{\Id,\delta\}$,
we are dealing with representations of the form $\rho_{(\mu, \mu)}$, with
$\mu \in \Fer(r', n')$. Notice that $\mu$ may be considered as belonging to $\Fer(r',1, p', n')$: acting on $\mu$ with an element of $C_{p'}$ corresponds to acting on $(\mu, \mu)$ with an element of $C_{p}$, and we know that
elements of $\Fer(r,n)$ in the same class modulo $C_p$ parametrizes the same irreducible representation of $G(r,p,n)$.
These $\rho_{(\mu, \mu)}$ are the representations of $G(r,n)$ that split as
$G(r,p,n)$-modules. As in the case of $D_n$, they split into two
different representations that we denote by $\rho^0_{[\mu,\mu]}$   and $\rho_{[\mu, \mu]}^{1}$. We also denote by   $\chi^0_{[\mu, \mu]}$ and $\chi^1_{[\mu,\mu]}$ the corresponding characters. Then the discrete Fourier transform of
$\rho_{\mu, \mu}$ is given by the two functions
$$\Delta_{\mu, \mu}^0(h)=\chi^0_{[\mu, \mu]}(h)+\chi_{[\mu, \mu]}^1(h); \qquad \Delta_{\mu, \mu}^1(h)=\chi_{[\mu, \mu]}^0(h)-\chi_{[\mu, \mu]}^1(h).$$
In the remaining of this section, we exploit equation \eqref{eccocosacivuole} to provide an explicit computation of the difference characters $\Delta^1_{\mu, \mu}$ for every $G(r,p,n)$ with $(p,n)=2$. This computation will turn up to be of crucial importance in the proof of Theorem \ref{AsymGrpn}.

Recall that when $\GCD(p,n)=2$, the conjugacy classes of $G(r,n)$ of the form $\mathrm{cl}_{2\alpha}$  split into two distinct $G(r,p,n)$-classes $\mathrm{cl}_{2\alpha}^0$ and $\mathrm{cl}_{2\alpha}^1$, where $2\alpha=(2\alpha^{(0)},\emptyset,2\alpha^{(2)},\emptyset,\ldots,2\alpha^{(r-2)},\emptyset)$.

\begin{notaz}
In what follows, we  often need to compute class functions on $G(r,p,n)$. For this reason, it will be useful to fix one special element, that we call \emph{normal}, for each $G(r,p,n)$-conjugacy class. This will be done as follows.
If the conjugacy class is not of the form ${\rm cl}_{2\alpha}^1$, the normal element  $h$ is defined as follows:
\begin{itemize}
\item the elements of each cycle of $h$ are chosen in increasing order, from the cycles of smallest color to the cycles of biggest color;
\item in every cycle of color $i$, all the elements have color $0$ but the biggest one whose color is $i$.
\end{itemize}
If the class is of the form ${\rm cl}_{2\alpha}^1$ the normal element $h$ is defined similarly with the unique difference that if the cycle containing $n$ has color $2j$ then the color of $n$ is $2j-1$ and the color of $n-1$ is 1.
For example, if 
\setlength{\unitlength}{6pt}
$$
2\alpha=\left( \begin{picture}(18,1.7)(0,0.7)               
\put(0,2){\line(1,0){2}} \put(5,2){\line(1,0){4}}  \put(12,2){\line(1,0){4}} 
\put(0,1){\line(1,0){2}} \put(5,1){\line(1,0){4}}  \put(12,1){\line(1,0){4}}
\put(12,0){\line(1,0){2}}
\put(0,1){\line(0,1){1}}\put(1,1){\line(0,1){1}}\put(2,1){\line(0,1){1}}\put(6,1){\line(0,1){1}}\put(7,1){\line(0,1){1}}
\put(8,1){\line(0,1){1}}\put(9,1){\line(0,1){1}}\put(5,1){\line(0,1){1}}\put(12,0){\line(0,1){2}}\put(13,0){\line(0,1){2}}
\put(14,0){\line(0,1){2}}\put(15,1){\line(0,1){1}}\put(16,1){\line(0,1){1}}
\put(2.2,1){,}\put(3.8,1){,}\put(9.2,1){,}\put(10.8,1){,}\put(16.2,1){,}
\put(3,1){$\emptyset$}
\put(10,1){$\emptyset$}
\put(17,1){$\emptyset$}
\end{picture}\right)
$$
Then the normal element in ${\rm cl}_{2\alpha}^0$ is $(1,2)(3,4,5,6^2)(7,8,9,10^4)(11,12^4)$ and the normal element in ${\rm cl}_{2\alpha}^1$ is $(1,2)(3,4,5,6^2)(7,8,9,10^4)(11^1,12^3)$, where we have omitted all the exponents equal to 0.
\end{notaz}

\begin{prop}\label{mu general case}
Let $g\in G(r,p,n)$, and $\mu\in \Fer(r',n')$. Let $\chi_\mu$ denote the character of the representation of $G(r',n')$ indexed by $\mu$. Then
\begin{equation*}
\Delta_{\mu,\mu}^1(g)=\left\{\begin{array}{ll}
                                              (-1)^\epsilon 2^{\ell(\alpha)}\chi_\mu(\alpha^{(0)},\alpha^{(2)}, \ldots, \alpha^{({r-2})}), &\textrm{if }g \in \mathrm{cl}_{2\alpha}^{\epsilon};\\
                                             
                                              0, &\mbox{if $g$ belongs to an unsplit conjugacy class}
                                             \end{array}
\right.,
\end{equation*}
where, if $\alpha=(\alpha^{(0)},\emptyset,\alpha^{(2)}, \emptyset,\ldots, \alpha^{({r-2})},\emptyset)$,  $\ell(\alpha)=\sum \ell(\alpha^{(i)})$. 
 \end{prop}

\begin{proof}
When $g$ does not belong to a split conjugacy class, $\Delta^1_{\mu, \mu}(g)=0$. In fact, $\chi^0_{[\mu, \mu]}$ and
$\chi^1_{[\mu, \mu]}$ are conjugate characters, so they must coincide on every element belonging to an unsplit class.

When $g$ does belong to a split conjugacy class, this proof consists of three steps:
\begin{enumerate}
\item Provide an explicit description for the $G(r,n)$-module $\rho_{\mu,\mu}$.
\item Build an associator $S$ for the $G(r,n)$-module $\rho_{\mu, \mu}$.
\item Compute the trace $\tr(S(g))=\Delta^1_{\mu, \mu}(g).$
\end{enumerate}

Let us start with the first step. For brevity, we set $\tau=(t_0,\ldots,t_{r'-1})$, where $t_i=|\mu^{(i)}|$ and $G(r,(\tau,\tau))\eqdef G(r,t_0)\times\cdots\times G(r,t_{r'-1})\times G(r,t_0)\times\cdots\times G(r,t_{r'-1})<G(r,n)$.

Our representation $\rho_{\mu,\mu}$ looks like this (see Theorem \ref{rapp di grn}):

\begin{eqnarray*}\rho_{\mu,\mu}&=&\Ind_{G(r,(\tau,\tau))}^{G(r,n)} \big(\tilde \rho_{\mu^{(0)}}\odot (\gamma_{n_1}\otimes \tilde\rho_{\mu^{(1)}})\odot  \cdots \odot (\gamma_{n_{r'-1}}^{\otimes r'-1}\otimes \tilde \rho_{\mu^{(r'-1)}})\odot(\gamma_{n_{0}}^{\otimes r'}\odot \tilde \rho_{\mu^{(0)}})\\ 
&& \odot(\gamma_{n_{0}}^{\otimes r'}\odot \tilde \rho_{\mu^{(0)}})\odot (\gamma_{n_1}^{\otimes r'+1}\otimes \tilde\rho_{\mu^{(1)}})\odot  \cdots \odot (\gamma_{n_{r'-1}}^{\otimes r-1}\otimes \tilde \rho_{\mu^{(r'-1)}} )\big)\\
&=&\Ind_{G\left(r,n'\right)\times G\left(r,n'\right)}^{G(r,n)}\Big( \Ind_{G(r,(\tau,\tau))}^{G\left(r,n'\right)\times G\left(r,n'\right)} 
\big(\tilde \rho_{\mu^{(0)}}\odot (\gamma_{n_1}\otimes \tilde\rho_{\mu^{(1)}})\odot  \cdots \odot (\gamma_{n_{r'-1}}^{\otimes r'-1}\otimes \tilde \rho_{\mu^{(r'-1)}})\\ 
&&\odot(\gamma_{n_{0}}^{\otimes r'}\odot \tilde \rho_{\mu^{(0)}})\odot (\gamma_{n_1}^{\otimes r'+1}\otimes \tilde\rho_{\mu^{(1)}})\odot  \cdots \odot (\gamma_{n_{r'-1}}^{\otimes r-1}\otimes \tilde \rho_{\mu^{(r'-1)}} )\big)\Big)\\
&=&\Ind_{G\left(r,n'\right)\times G\left(r,n'\right)}^{G(r,n)} \rho_\mu \odot (\gamma_{n'}^{r'}\otimes \rho_\mu )\\
&=&\Ind_{G\left(r,n'\right)\times G\left(r,n'\right)}^{G(r,n)} \rho_\mu \odot (\delta\otimes \rho_\mu ).
\end{eqnarray*}
We need to give an explicit description of this representation of $G(r,n)$. We consider the set $\Theta$ of two-rowed arrays $\left[\begin{array}{lll} t_1& \ldots & t_{n'}\\ t_{n'+1}&\ldots&t_n \end{array} \right]$
such that $\{t_1, \ldots, t_n\}=\{1, \ldots,n\}$ and the $t_i$'s increase on each of the two rows. Each element in $\Theta$ can be identified with the permutation whose window notation is $[t_1, \ldots, t_{n}]$.
\begin{prop}\label{coset}
 Let $g \in G(r,n)$.  Then $$\exists \,\,!\,\, t'\in \Theta : g=t'(x_1, x_2) \mbox{ with }
(x_1, x_2)\in G\left(r,n'\right) \times G\left(r,n'\right).$$
\end{prop}

\begin{proof}
Existence.
Let $g=[\sigma_1^{z_1}, \ldots, \sigma_n^{z_n}]$, and let $t$ be the tabloid whose first and second line are filled with the (reordered) integers $\sigma_1, \ldots, \sigma_{n'}$, and $\sigma_{n'+1}, \ldots, \sigma_n$ respectively. Since we need to obtain $g=t(x_1, x_2)\mbox{ with }
(x_1, x_2)\in G\left(r,n'\right) \times G\left(r,n'\right)$, we have to check that $t^{-1}g \in G\left(r,n'\right) \times G\left(r,n'\right)$, i.e.,
\begin{align*}
1 \leq t^{-1}|g|(i)=t^{-1}\sigma_i\leq n', \quad  &\mbox{ if } i\in [n']\\
n'<t^{-1}|g|(i)t^{-1}\sigma_i\leq n, \quad &\mbox{ if }i\in [n'+1,n],
\end{align*}
and this is an immediate consequence of the way $t$ was constructed.

Uniqueness. This is due to cardinality reasons:
$$\left\rvert\Theta\right\rvert\left\rvert G\left(r,n'\right)\times G\left(r,n'\right)\right\rvert=\binom{n}{n'}
(n'!r^{n'})^2=\left\rvert G(r,n)\right\rvert.$$
\end{proof}

Thanks to Proposition \ref{coset}, a set of coset representatives for  $${G(r,n)}/({G(r,n')\times G(r,n')})$$ is given by $\Theta$.
Let $T$ be the vector space spanned by the elements of $\Theta$.
 The vector space associated to the representation we are dealing with can be identified with $T\otimes V\otimes V$, and the action of $\rho_{\mu, \mu}$ on it is given by
\begin{align*}
 \rho_{\mu, \mu}:G(r,n)&\longrightarrow GL(T \otimes V \otimes V)\\
x&\longmapsto  \rho_{\mu, \mu}(x):T \otimes V \otimes V \longrightarrow T \otimes V \otimes V\\
& \qquad \qquad \quad\,\,  \,\,\,\,\,t\otimes v_1\otimes v_2 \longmapsto \delta(x_2) t' \otimes \rho_{\mu}(x_1) (v_1) \otimes \rho_{\mu}(x_2)( v_2),
\end{align*}
where $t'$, $x_1$ and $x_2$ are uniquely determined by the relation $xt=t'(x_1, x_2)$ with $t\in \Theta$, $(x_1, x_2) \in G\left(r,n'\right) \times G\left(r,n'\right)$.

We are now ready for the second step.
\begin{prop}
The automorphism $S\in \GL(T \otimes V \otimes V)$ so defined:
$$S(t\otimes v_1\otimes v_2)=\hat t \otimes v_2\otimes v_1, $$
where $\hat t$ is the element of $\Theta$ obtained from $t$ by exchanging its rows,
is an associator for $T \otimes V \otimes V$.
\end{prop}
\begin{proof} 
All we have to show is that $S$ is an isomorphism of representations between $\rho_{\mu, \mu}$ and
$ \delta\otimes \rho_{\mu, \mu}$, i.e.
$$S\circ\rho_{\mu, \mu}(g)=\delta(g)\rho_{\mu, \mu}(g)\circ S.$$
The set of permutations together with the diagonal matrix $[1^0, 2^0, \ldots, n^1]$ generate $G(r,n)$, so this verification can be accomplished when $g$ is one of these elements only. Just exploit Proposition \ref{coset} to write $gt$ as a product $t'(g_1, g_2)$, and notice that
 $\hat t= ts$ if $s$ is the tabloid $\left[\begin{array}{lll}
n'+1& \ldots &  n\\
              1&\ldots&n'
          \end{array}
\right]$. Also, when $g=[1^0, 2^0, \ldots, n^1]$, the equality 
$g\widehat{t}=gts=t(g_1, g_2)s=ts(g_2,g_1)=\widehat{t}(g_2,g_1)$ is used. 
\end{proof}

Finally, the last step: let us compute $\Delta^1_{\mu, \mu}(g)$, for every
$g$ belonging to a split conjugacy class. Since $\Delta^1_{\mu, \mu}$ is a class function, we can choose one special element from each $G(r,p,n)$-conjugacy class. In fact, even less is needed: it will be enough to choose one element from each $G(r,n)$-conjugacy class, because of the useful relation (see \cite{Stembridge}, Proposition 6.2)
\begin{equation}\label{formulina}
\Delta^1_{\mu,\mu}(ghg^{-1})= \delta(g)\Delta^1_{\mu,\mu}(h) \,\,\, \forall g \in G(r,n), h\in G(r,p,n).
\end{equation}
We will compute $\Delta^1_{\mu, \mu}(h)$, where $h$ is the normal element of the class $\mathrm{cl}_{2\alpha}^0$. 
By definition,
$\Delta^1_{\mu, \mu}(h)=\tr(Sh|_{T\otimes V\otimes V})$.
Now,
 given $t'\in \Theta $ and $(h_1, h_2)\in G\left(r, n'\right)\times G\left(r, n'\right)$ verifying $ht=t'(h_1, h_2)$, if
 $v_i$, $v_j$ are vectors of a basis of $V$,
\begin{align*}
S [h (t \otimes v_i \otimes v_j )]&= S [\delta(h_2) t' \otimes h_1 v_i \otimes h_2 v_j]\\
&= \delta(h_2) \hat {t'} \otimes h_2 v_j \otimes h_1 v_i\\
&= \hat{ t'}\otimes h_2 v_j \otimes h_1 v_i,
\end{align*}
where the last equality depends on the special way we chose $h$. Namely, since $ (h_1, h_2)=(t')^{-1}ht $ with $t, t'\in S_n$, the colors of $(h_1, h_2)$ are the same as in $h$ and they are simply permuted, so in $(h_1, h_2)$ the colors are all even. So the trace we are computing is given by

$$\sum_{i,j=1, \ldots, n', t=\hat{ t'}}(\rho_{\mu}(h_2))_{i,j}(\rho_{\mu}(h_1))_{j,i}=\sum_{ t=\hat{ t'}}\chi^\mu(h_1 h_2).$$
Recall the way $t'$ is constructed in the proof of Proposition \ref{coset}: $t'=\hat t$ if and only if $|h|(t_i)$ belongs to $\{t_{n'+1}, \ldots, t_{2n}\}$ for every $i\in [n']$:
$$\{t_{n'+1}, \ldots, t_{2n}\}=\{|h|(t_i)\}_{1 \leq i \leq n'}, $$ and, viceversa, $$\{t_1, \ldots, t_{n'}\}=\{|h|(t_i)\}_{n' <i \leq n}. $$
So the $t$'s satisfying $t=\widehat{t'}$ are those $t$ such that, for every cycle of $h$, the consecutive numbers are in opposites rows. We have two possibilities for each cycle, so they are $2^{\ell(\alpha)}$.

Furthermore, suppose $h$ contains a cycle of length $2k$ and color $2j$. Then, according to which of the two possible choices is made for $t$, a cycle of length $k$ and color $j$ will be contained either in $h_1$ or in $ h_2$. Thus, $h_1h_2$ belongs to the $G\left(r',n'\right)$-class $\mathrm{cl}_{\alpha^{(0)},\alpha_{(2)}, \ldots, \alpha_{(r-2)}}$.
So our final result is
$$\Delta^1_{\mu, \mu}(h)=2^{\ell(\alpha)}\chi^\mu(\alpha^{(0)},\alpha^{(2)}, \ldots, \alpha^{(r-2)}).$$

Let us now turn to the elements belonging to the other split conjugacy class $\mathrm{cl}^1_{2\alpha}$. 
If $h$ belongs to $\mathrm{cl}^0_{2\alpha}$, thanks to Lemma \ref{signature}, 
$$ghg^{-1}\in \mathrm{cl}^1_{2\alpha}\Rightarrow z(g)=1 \mod 2 \Rightarrow \delta(g)=-1,$$ therefore (see equality \eqref{formulina})
$$\Delta^1_{\mu, \mu}(ghg^{-1})=\delta(g)\Delta^1_{\mu, \mu}(h) =-2^{\ell(\alpha)}\chi^\mu(\alpha^{(0)},\alpha^{(2)}, \ldots, \alpha^{(r-2)}).$$
\end{proof}

\section{The antisymmetric submodule}\label{Asym}

This section is the real heart of this paper. We study the irreducible decomposition of the antisymmetric submodule $M^1$ (and hence also of the symmetric submodule $M^0$) of the Gelfand model $M(r,1,p,n)$ of the group $G(r,p,n)$ constructed in Theorem \ref{modello}. More precisely we will show that  $M^1$ affords exactly one representation of each pair of split irreducible representations of $G(r,p,n)$ namely, the one labelled with  $1$. \\
 If $\GCD(p,n)=1$ the antisymmetric submodule vanishes (and there are no split representations) so in this section we can always assume that $\GCD(p,n)=2$.

\begin{thm}\label{AsymGrpn}
Let $M^1$ be the antisymmetric submodule of the Gelfand model $M(r,1,p,n)$ of $G(r,p,n)$. Then $$(M^1,\varrho)\cong \bigoplus_{[\mu,\mu]\in \Fer(r,1,p,n)}\rho_{[\mu, \mu]}^1.$$
\end{thm}

\begin{proof}
The strategy in this proof is the one outlined for the case of Weyl groups of type $D$. So we consider the two representations of $G(r,p,n)$ $$(M^1,\phi^0)\quad \mbox{ and }\quad (M^1,\phi^1),$$ given by
$$\phi^0(g)(C_v)\eqdef\zeta_r^{\langle g,v\rangle}C_{|g|v|g|^{-1}},  \qquad \phi^1(g)(C_v)\eqdef\zeta_r^{\langle g,v\rangle}\zeta_r^{a(g,v)}C_{|g|v|g|^{-1}}$$
(notice that $\phi^1(g)=\varrho(g)|_{M^1}$). 
The main idea of this proof is to exploit Proposition \ref{mu general case} to show that
\begin{equation}\label{piumeno}
\chi_{\phi^0}(g)-\chi_{\phi^1}(g)=\sum_{[\mu]\in \Fer(r',1,p',n')}\chi_{[\mu, \mu]}^0(g)-\sum_{[\mu]\in \Fer(r',1, p',n')}\chi_{[\mu, \mu]}^1(g),
\end{equation}
where we observe that if $[\mu]$ ranges in $\Fer(r',1,p',n')$ then $[\mu,\mu]$ ranges in $\Fer(r,1,p,n)$. First of all, we will compute the right-hand side of \eqref{piumeno}. We already know that it vanishes on every $g$ belonging to an unsplit conjugacy class. So let $g\in \mathrm{cl}_{2\alpha}^0$.

Let $\chi_M$ denote the character of a model for the group $G(r',n')$. Then
\begin{align*}
\sum_{[\mu]\in \Fer(r',1,p',n')}(\chi_{[\mu, \mu]}^0-\chi_{\mu, \mu}^1)(g)&=
\frac{1}{p'}\sum_{\mu\in \Fer(r',n')}(\chi_{[\mu, \mu]}^0-\chi_{[\mu, \mu]}^1)(g)\\
&=\frac{1}{p'}\sum_{\mu\in \Fer(r',n')}2^{\ell(\alpha)}\chi_\mu(\alpha)\\
&=\frac{1}{p'}2^{\ell(\alpha)} \chi_M(\alpha),
\end{align*}
where the first equality holds because the contribution of every $\mu\in \Fer(r',n')$ provides $p'$ copies of the same irreducible representation of $G(r,p,n)$, the second one follows from Proposition \ref{mu general case}, and $\chi_M(\alpha)$ denotes the value of the character $\chi_M$ on any element of the conjugacy class ${\rm cl}_{\alpha}$ of $G(r',n')$.

Now, following \cite{Ca2}, for $g\in G(r,n)$ we denote by $\Pi^{2,1}(g)$ the set of partitions of the set of disjoint cycles of $g$ into:
 \begin{itemize}
   \item singletons;
   \item pairs of cycles having the same length.
 \end{itemize}
If $\pi\in\Pi^{2,1}(g)$ we let $\ell(\pi)$ be the number of parts of $\pi$ and  $\pair_j(\pi)$ be the number of parts of $\pi$ which are pairs of cycles of length $j$. Moreover, if $s\in \pi$ is a part of $\pi$ we let $z(s)$ be the sum of the colors of the (either 1 or 2) cycles in $s$.\\
If $g$ and $g'$ belong to the same conjugacy class  ${\rm cl}_\alpha$ there is clearly a bijection between $\Pi^{2,1}(g)$ and $\Pi^{2,1}(g')$ preserving the statistics $\ell(\pi)$ and $\pair_j(\pi)$, and  the colors $z(s)$ of the parts of $\pi$; therefore we sometimes write $\Pi^{2,1}(\alpha)$ meaning   $\Pi^{2,1}(g)$ for some $g\in {\rm cl}_\alpha$.\\
The following equality holds (see \cite[Proposition 3.6]{Ca2}):
$$ \chi_M(\alpha)=\sum_{\pi}\left(r'\right)^{\ell(\pi)}\prod_jj^{\pair_j(\pi)}\\
$$
for all $\alpha\in \Fer(r',n')$, where the sum is taken over all elements of $\Pi^{2,1}(\alpha)$ having no singletons of even length and such that $z(s)=0\in \mathbb Z_{r'}$ for all $s \in \pi$.

Let us now evaluate $\chi_{\phi^0}(g)-\chi_{\phi^1}(g)$. To this aim, we recall some further notation used in \cite{Ca2}. Consider, for every $g\in G(r,p,n)$ and $\epsilon\in \mathbb Z_2$, the set
$$A^{\epsilon}(g):=\{w\in G(r,n):w \textrm{ is antisymmetric and } |g|w|g|^{-1}=(-1)^\epsilon w\}.$$
 Any $w\in A^{\epsilon}(g)$ determines a partition $\pi(w)\in \Pi^{2,1}(g)$: a cycle $c$ is a singleton of $\pi(w)$ if the restriction of $|w|$ to $\Supp(c)$ is a permutation of $\Supp(c)$ and $\{c,c'\}$ is a pair of $\pi(w)$ if the restriction of $|w|$ to $\Supp(c)$ is a bijection between $\Supp(c)$ and $\Supp(c')$. Finally, if $\pi\in \Pi^{2,1}(g)$, we let $A^{\epsilon}_\pi\eqdef \{w \in A^\epsilon(g): \pi(w)=\pi\}$. Then the set $A^\epsilon(g)$ can be decomposed
into the disjoint union
\begin{equation}\label{unionona}
A^\epsilon(g)=\bigcup_{\pi \in \Pi^{2,1}(g)}A^{\epsilon}_\pi.
\end{equation}

With the above notation, we have
$$\chi_{\phi^0}(g)-\chi_{\phi^1}(g)=\frac{1}{p}\sum_{\pi \in \Pi^{2,1}(g)}\sum_{A_\pi^0 \cup A_\pi^1}\zeta_r^{\langle g,w\rangle}(1-\zeta_r^{a(g,w)}).$$
Since (see \cite[Lemma 5.7] {Ca2}) 
$$
a(g,w)=\left\{\begin{array}{ll}0,&\textrm{if }w\in A_\pi^0(g);\\r',&\textrm{if }w\in A_\pi^1(g),
              \end{array}\right.$$ we find
\begin{align}\label{halfway}
\chi_{\phi^0}(g)-\chi_{\phi^1}(g)&= \frac{1}{p'}\sum_{\pi \in \Pi^{2,1}(g)}\sum_{w \in A_\pi^1}\zeta_r^{\langle g,w\rangle}.
\end{align}

If $\pi=\{s_1,\ldots,s_h\}$, it is shown in \cite[Section 5]{Ca2} that the set $A_\pi^1$ has a natural decomposition $A_\pi^1=A_{s_1}^1\times \ldots \times A_{s_h}^1$, i.e. every $w$ in $A_\pi^1$ can be written as a product $w= w_1 \cdot \ldots \cdot w_h$, with $w_i\in A_{s_i}^1$. The sets $A_{s_i}^1$ depend on the structure of $|s_i|$ only, and are described in Table \ref{rotfloat2}. In this table the indices of $i_1,\ldots,i_d,j_1,\ldots,j_d$ should be considered in $\mathbb Z_d$ and in any box of the table the parameters $k\in \mathbb Z_r$ and $l\in \mathbb Z_d$ are arbitrary but fixed. 
\begin{table}
\renewcommand{\arraystretch}{1.5}
\begin{tabular}{|c|l|}
\hline
$|s|$&$A^1_s$\\
\hline
\hline
\hspace{-1.8cm}$\{(i_1,\ldots,i_d)\}$  &\raisebox{-3mm}[0pt][0pt]{$i_h\mapsto\zeta_r^k (-1)^hi_{h+\frac{d}{2}}$}\\
\hspace{1cm}with $d\equiv 2 \mod 4$& \\
\hline
\hspace{-1.8cm}$\{(i_1,\ldots,i_d)\}$&\raisebox{-3mm}[0pt][0pt]{$\emptyset$}\\
\hspace{1cm}with $d\not \equiv 2 \mod 4$&\\
\hline
$\{(i_1,\ldots,i_d),(j_1,\ldots,j_d)\}$,  &\raisebox{-3mm}[0pt][0pt]{$\emptyset$}\\
\hspace{1cm}with $d$ odd & \\
\hline
$\{(i_1,\ldots,i_d),(j_1,\ldots,j_d)\}$,&$ i_h\mapsto \zeta_r^k (-1)^h j_{h+l}$\\
\hspace{1cm} with $d$ even & and $j_h\mapsto-\zeta_r^k (-1)^{h-l} i_{h-l}$\\
\hline
\end{tabular}
\vspace{5mm}
\caption[]{}\label{rotfloat2}
\end{table}
For example, if $s=\{(1,2),(3,4)\}$, and $r=4$, then $A_s^1$ consists of the 8 elements having either the form $(1^{k+2},3^k)(2^k,4^{k+2})$ or the form $(1^{k+2},4^k)(2^k,3^{k+2})\}$, as $k$ varies in $\{0,1,2,3\}$. This allows to focus on the single sets $A_s^1$, via the identity
\begin{align*}
\sum_{w\in A_\pi^1}\zeta_r^{\langle g,w\rangle} &=
\sum_{w_1\in A_{s_1}^1,\ldots ,w_h\in A_{s_h}^1}\zeta_r^{\langle g,w_1\cdots w_h\rangle}\\
&=\sum_{w_1\in A_{s_1}^1,\ldots ,w_h\in A_{s_h}^1}\zeta_r^{\sum_i \langle g_i,w_i\rangle}\\
&=\prod_{i=1}^h \sum_{w_i\in A_{s_i}^1}\zeta_r^{\langle g_i,w_i\rangle},
\end{align*}
where $g_i\in G(r,\Supp(s_i))$ is the restriction of $g$ to $\Supp(s_i)$ (if $s\in\Pi^{2,1}(g)$, we let $\Supp(s)$, the support of $s$, be the union of the supports of the cycles in $s$).

\begin{lem}\label{Fpi0}
If $g\in G(r,n)$ has at least one cycle $c$ of odd length, then
$$
\sum_{w \in A_\pi^1}\zeta_r^{\langle g,w\rangle}=0
$$
for all $\pi\in \Pi^{2,1}(g)$.
\end{lem}
\begin{proof}
This is trivial since in these cases $A_\pi^1=\emptyset$ (see Table \ref{rotfloat2}).
\end{proof}
By Lemma \ref{Fpi0} we can restrict our attention to those elements $g$ having all cycles of even length, and to those partitions $\pi\in \Pi_{2,1}(g)$ whose singletons have length $\equiv 2 \mod 4$.

\begin{lem}\label{cyc1}
If $g\in G(r,p,n)$ has at least one cycle $c$ of odd color, then
$$
\sum_{\pi\in \Pi_{2,1}(g)}\sum_{w \in A_\pi^1}\zeta_r^{\langle g,w\rangle}=0.
$$
for all $\pi\in \Pi_{2,1}(g)$.
\end{lem}
\begin{proof} Since the left-hand side is a class function we can assume that $g$ is normal. We prove in this case the stronger statement that $\sum_{w \in A_\pi^1}\zeta_r^{\langle g,w\rangle}=0$ for all $\pi\in \Pi_{2,1}(g)$. By Lemma \ref{Fpi0}, we can assume that the cycle $c$ has even length.
We split this result into two cases.
Assume that the cycle of odd color - say $j$ - is a singleton $s_i=\{c\}$ of $\pi$.  Then Table \ref{rotfloat2} furnishes the structure of $A_{s_i}^1$, and we find 
\begin{align*}
\sum_{w\in A_{s_i}^1}\zeta_r^{\langle g_i,w_i\rangle}=\sum_{k=0}^{r-1}\zeta_r^{jk}=0, 
\end{align*}
since $j$ is odd and cannot be a multiple of $r$. 

Now 
assume that the cycle $c$ belongs to a pair $s_i$ of $\pi$. Let us call $a$ and $b$ the two colors of the cycles in $s_i$, with $b$ odd. 
Again, looking at Table \ref{rotfloat2}, 
\begin{align*}
\sum_{w_i\in A_{s_i}^1}\zeta_r^{\langle g_i,w_i\rangle}&=\sum_{k=0}^{r-1}\sum_{l=0}^{d-1}\zeta_r^{ak+b(k+(l+1)r')}\\
&=\sum_{k=0}^{r-1}\big(\frac{d}{2}\zeta_r^{(a+b)k}+\frac{d}{2}\zeta_r^{ak+b(k+r')}\big)\\
&=\frac{d}{2}\sum_{k=0}^{r-1}\zeta_r^{(a+b)k}(1+\zeta_r^{br'})\\
&=\frac{d}{2}(1+\zeta_r^{br'})\sum_{k=0}^{r-1}\zeta_r^{(a+b)k}.
\end{align*}
Since $b$ is odd, the factor $1+\zeta_r^{br'}$ vanishes and so does the whole sum. 
\end{proof}

\begin{lem}\label{caso}
Let $g\in G(r,p,n)$ be normal and such that all cycles of $g$ have even color and even length. Then, for all $\pi\in \Pi^{2,1}(g)$,
$$\sum_{w \in A_\pi^1}\zeta_r^{\langle g,w\rangle}=\left\{\begin{array}{ll}
                                               (-1)^{\sign(g)}|A_\pi^1|, &\mbox{if $z(s)=0$ for all $s\in \pi$;}\\
                                               0, &\mbox{ otherwise.}
                                              \end{array}\right.
$$
\end{lem}
\begin{proof}
We first assume that $\sign(g)=0$.
If $s_i$ is a singleton of $\pi$ of color $2j$ and length of $\not \equiv 2 \mod 4$, then $A_{s_i}$ is empty and the result clearly follows. So we can assume that $\ell(s_i)\equiv 2 \mod 4$ and we can derive the value of $\langle g_i, w_i\rangle$ from Table \ref{rotfloat2}, and we obtain
\begin{align*}
\sum_{w_i\in A_{s_i}^1}\zeta_r^{\langle g_i,w_i\rangle}=\sum_{k=0}^{r-1}\zeta_r^{2jk}=\left\{\begin{array}{ll}
                                                                                 0,& \mbox{ if }2j\not \equiv 0 \mod r;\\
\\
                                                                                 r=|A_{s_i}^1|, &\mbox{ if }2j\equiv 0 \mod r.
                                                                                \end{array}\right.
\end{align*}

Let now $s_i=\{c_1,c_2\}$ be a pair of cycles of length $d$, and colors respectively $2a$ and $2b$.
\begin{align*}
\sum_{w_i\in A_{s_i}^1}\zeta_r^{\langle g_i,w_i\rangle}&=\sum_{k=0}^{r-1}\sum_{l=0}^{d-1}\zeta_r^{2ak+2b(k+(l+1)r')})\\
&=d\sum_{k=0}^{r-1}\zeta_r^{2ak+2bk}=\left\{\begin{array}{ll}
                                                                                 0, &\mbox{ if }2a+2b\not \equiv0 \mod r;\\
\\
                                                                                 dr=|A_{s_i}^1|, &\mbox{ if }2a+2b \equiv 0 \mod r.
                                                                                \end{array}\right.
\end{align*}
If $\sign(g)=1$ the proof is similar and is left to the reader.
\end{proof}
We are now ready to prove Theorem \ref{AsymGrpn}. Let first $g$ belong to an unsplit class. Then $g$ has a cycle of odd length or a cycle of odd color, and then Lemmas \ref{Fpi0} and \ref{cyc1} ensure that 
\begin{equation} \label{uns}                                                                                                                                                                                                                                                                                                                                                                                                                    \chi_{\phi^0}(g)-\chi_{\phi^1}(g)=0=\sum_{[\mu]\in \Fer(r',1,p',n')}\chi_{[\mu, \mu]}^0-\sum_{[\mu]\in \Fer(r',1, p',n')}\chi_{[\mu, \mu]}^1(g).
\end{equation}

So let $g$ belong to the split conjugacy class of the form $c_{2\alpha}^\eta$. We are interested in the evaluation of the sum appearing in \eqref{halfway} and so we can assume that $g$ is also a normal element.
Thanks to  Lemma \ref{caso}, the only partitions $\pi\in \Pi^{2,1}(g)$ contributing to the sum \eqref{halfway} are those verifying $z(s)=0 \mod r$ for all $s\in \pi$.
Thus,
\begin{align*}
\chi_{\phi^0}(g)-\chi_{\phi^1}(g)&=\frac{1}{p'}(-1)^\eta\sum_{\pi \in \Pi^{2,1}(g)}|A_\pi^1|\\
&=\frac{1}{p'}(-1)^\eta\sum_{\pi \in
\Pi^{2,1}(g)}r^{\ell(\pi)}\prod_j(2j)^{\pair_{2j}(\pi)},
\end{align*}
where the sum is taken over all partitions of $\Pi^{2,1}(g)$ such that:
\begin{itemize}
 \item singletons have length $\equiv 2 \mod 4$ (recall that $A_\pi^-=\emptyset$ if $\pi$ has a singleton of length $\equiv 0 \mod 4$);
 \item pairs have even length;
\item $z(s)=0 \mod r$, for all $s\in \pi$.
\end{itemize}

Summarizing, if $g\in {\rm cl}_{2\alpha}^\eta$, we have
\begin{align*}
\chi_{\phi^0}(g)-\chi_{\phi^1}(g)&= \frac{1}{p'}(-1)^\eta\sum_{\pi\in \Pi^{2,1}(2\alpha)}r^{\ell(\pi)}\prod_j(2j)^{\pair_{2j}(\pi)}\\
&= \frac{1}{p'}(-1)^\eta\sum_{\pi\in \Pi^{2,1}(2\alpha)}\left(2r'\right)^{\ell(\pi)}\prod_j(2j)^{\pair_{2j}(\pi)}\\
&=\frac{1}{p'}(-1)^\eta2^{\ell(\pi)+\sum_j\pair_{2j}(\pi)} \sum_{\pi\in \Pi^{2,1}(2\alpha)}\left(r'\right)^{\ell(\pi)}\prod_jj^{\pair_{2j}(\pi)}\\
&=\frac{1}{p'}(-1)^\eta2^{\ell(\alpha)} \sum_{\pi\in \Pi^{2,1}(\alpha)}\left(r'\right)^{\ell(\pi)}\prod_jj^{\pair_j(\pi)},\\
\end{align*}
where $\alpha$ has to be considered as an element in $\Fer(r',n')$ and the last sum is taken over all partitions of $\Pi^{2,1}(\alpha)$ whose singletons have odd length (and pairs have  any length), and $z(s)=0\in \mathbb Z_{r'}$ for all $s\in \pi$.

The above computation, together with \eqref{uns} and Proposition \ref{mu general case},
leads to
\begin{equation}\label{4caratteri}
\sum_{[\mu]\in \Fer(r',1,p',n')}\chi_{[\mu, \mu]}^0(g)+\chi_{\phi^1}(g)=
\sum_{[\mu]\in \Fer(r',1,p',n')}\chi_{[\mu, \mu]}^1(g)+\chi_{\phi^0}(g) \qquad \forall \,g \in G(r,p,n).
\end{equation}
Now, $\sum_{\mu\in \Fer(r',1,p',n')}\chi_{\mu, \mu}^0$ and $\sum_{\mu\in \Fer(r',1,p',n')}\chi_{\mu, \mu}^1$ are orthogonal characters. As
\begin{equation*}
\sum_{\mu\in \Fer(r',1,p',n')}\dim(\rho^0_{[\mu, \mu]})=\sum_{\mu\in \Fer(r',1,p',n')}|\St_{[\mu, \mu]}|=\dim(M^1) =\dim (\phi^0),
\end{equation*}
and, analogously, $\sum_{\mu\in \Fer(r',1,p',n')}\dim(\rho^1_{[\mu, \mu]})=\dim (\phi^1)$, we can conclude that
$$\sum_{\mu\in \Fer(r',1,p',n')}\chi^0_{[\mu,\mu]}(g)=\chi_{\phi^0}(g) \quad \mbox{ and }\quad \sum_{\mu\in \Fer(r',1,p',n')}
\chi^1_{[\mu,\mu]}(g)=\chi_{\phi^1}(g).$$ 
Recalling that $\rho^1(g)=\varrho(g)|_{M^1}$, the above equality means that
$$(M^1,\varrho)\cong \bigoplus_{\mu \vdash m}
\rho^1_{[\mu, \mu]},$$ and Theorem \ref{AsymGrpn} is proved.
\end{proof}

\section{The antisymmetric classes}\label{anclas}
An antisymmetric element of $G(r,n)$ can be characterized by the structure of its cycles, namely, an element $v\in G(r,n)$ is antisymmetric if and only if every cycle $c$ of $v$ has length 2 and is of the form $c=(a_1^{z_{a_1}},a_2^{z_{a_2}})$ with $z_{a_2}=z_{a_1}+r'$. We say that the residue class of $z_{a_1}$ and $z_{a_2}$ modulo $r'$ is the \emph{type} of $c$.   If the number of disjoint cycles of type $i$ of an antisymmetric element $v$ of $G(r,n)$ is $t_i$, then the  integer vector $\tau(v)=(t_0,\ldots,t_{r'-1})$ is called the \emph{type} of $v$. It is clear that two antisymmetric elements in $G(r,n)$ are $S_n$-conjugate if and only if they have the same type. We denote by $AC(r,n)$ the set of types of antisymmetric elements in $G(r,n)$, i.e. the set of vectors $(t_0,\ldots,t_{r'})$ with nonnegative integer entries such that $t_0+\cdots+t_{r'-1}=n'$.  If $\GCD(p,n)=2$  we let  $\gamma$ be the cyclic permutation of $AC(r,n)$ defined by $\gamma(t_0,\ldots,t_{r'-1})=(t_{r/p},t_{1+r/p}\ldots,t_{r'-1+r/p})$ where the indices must be intended as elements in $\mathbb Z_r$. We observe that $\gamma$ has order $p'$ and so we have an action of the cyclic group $C_p'$ generated by $\gamma$ on $AC(r,n)$.  We denote the quotient set by $AC(r,p,n)^*$. The type of an antisymmetric element of $G(r,p,n)^*$ is then an element of $AC(r,p,n)^*$ and if $[\tau]\in AC(r,p,n)^*$ we let $c^{1}_{[\tau]}$ be the $S_n$-conjugacy class consisting of the antisymmetric absolute involutions in $G(r,p,n)^*$ of type $[\tau]$. The main result of this section, Theorem \ref{mainanticlass}, provides a compatibility between the coefficients of $\tau$ and the sizes of the indices of the irreducible components of the module $M(c^1_{[\tau]})$. For this, it will be helpful the following criterion.
\begin{prop}\label{crit}
   Let $\nu=(n_0,\ldots,n_{r-1})$ be a composition of $n$ into $r$ parts, and $\rho$ a representation of $G(r,n)$. Then the following are equivalent:
\begin{enumerate}
 \item The irreducible subrepresentations of $\rho$ are all of the form $\rho_{\lambda^{(0)}, \ldots, \lambda^{(r-1)}}$ with $|\lambda^{(i)}|=n_i$ for all $i\in [0,r-1]$;
\item There exists a representation $\phi$ of $G(r,\nu)$ such that $\rho=\Ind_{G(r,\nu)}^{G(r,n)}(\phi)$ and $\phi(g)=\zeta_r^{\sum iz(g_i)}\phi(|g|)$, for all $g=(g_1,\ldots,g_{r-1})\in G(r,\nu)$.
\end{enumerate}
\end{prop}
\begin{proof}
 In proving that (1) implies (2) we can clearly assume that $\rho$ is irreducible and in this case the result is straightforward from the description in Proposition \ref{rapp di grn}.
In proving that (2) implies (1) we can assume that $\phi$ is irreducible. Then it is clear that $\phi\downarrow_{S_{\nu}}$ is also irreducible, where $S_\nu=S_{n_0}\times \cdots \times S_{n_{r-1}}$. In particular there exist $\lambda^{(0)},\ldots, \lambda^{(r-1)}$, with  $|\lambda^{(i)}|=n_i$ such that $\phi\downarrow_{S_{\nu}}\cong \rho_{\lambda^{(0)}}\odot\cdots\odot \rho_{\lambda^{(r-1)}}$. Now we can conclude that
$$
\phi\cong (\gamma_{n_0}^0\otimes \tilde \rho_{\lambda^{(0)}})\odot \cdots \odot (\gamma_{n_{r-1}}^{r-1}\otimes \tilde \rho_{\lambda^{(0)}}),
$$
and so the result follows again from Proposition \ref{rapp di grn}.
\end{proof}

We now concentrate on the special case $p=2$, so that $p'=1$. The general case will then be a direct consequence. Since the index of $G(r,2,n)$ in $G(r,n)$ is 2 the induction $\psi=\Ind_{G(r,2,n)}^{G(r,n)}(M(c^1_{\tau}),\varrho)$ of  the $G(r,2,n)$-representation $M(c^1_{\tau})$ to $G(r,n)$ is a representation on the direct sum $V\oplus V'$ of two copies of $V\eqdef M(c^1_{\tau})$.  So  a basis of $V\oplus V'$ consists of all the elements $C_v$, $C'_v$, as  $v$ varies in $c^1_\tau$. If $x=[1^{r-1},2^0,\ldots,n^0]$ is taken as a representative of the nontrivial coset of $G(r,2,n)$ and we impose that $x\cdot C_v=C'_v$, the representation $\psi$ of $G(r,n)$ on $V\oplus V'$ will be as follows
 
$$
g\cdot C_v=\left\{ \begin{array}{ll} \zeta_r^{\langle g,\tilde v\rangle}\zeta_r^{z_1(\tilde v)-z_{|g|^{-1}(1)}(\tilde v)}C_{|g|v|g|^{-1}} &  \textrm{if }g\in G(r,2,n),\\ \zeta_r^{\langle g,\tilde v\rangle}\zeta_r^{z_1(\tilde v)}C'_{|g|v|g|^{-1}}& \textrm{if }g\notin G(r,2,n),\end{array}\right.
$$
and
$$
g\cdot C'_v=\left\{ \begin{array}{ll} \zeta_r^{\langle g,\tilde v\rangle}C'_{|g|v|g|^{-1}} &  \textrm{if }g\in G(r,2,n),\\ \zeta_r^{\langle g,\tilde v\rangle}\zeta_r^{-z_{|g|^{-1}(1)}(\tilde v)}C_{|g|v|g|^{-1}}& \textrm{if }g\notin G(r,2,n);\end{array}\right.
$$
where $\tilde v$ is any lift of $v$ in $G(r,n)$.
Now we want to show that this representation $\psi$ of $G(r,n)$ is actually also induced from a particular representation of $G(r,(\tau,\tau))\eqdef G(r,t_0)\times\cdots G(r,t_{r'-1})\times G(r,t_0)\times\cdots G(r,t_{r'-1})$. With this in mind we  let $\mathcal C$ be the set of elements $v\in c^1_\tau$ having a lift $\tilde v$ in $G(r,n)$ satisfying the following condition: if $(a^i,b^{i+r'})$ is a cycle of $\tilde v$ of type $i$, then $a\in [t_0+\cdots+t_{i-1}+1,t_0+\cdots+t_{i}]$ and $b\in [n'+t_0+\cdots+t_{i-1}+1,n'+t_0+\cdots+t_{i}]$. Then, if $z\eqdef\min\{j:\,t_j\neq 0\}$ we let 
$$W\eqdef \bigoplus_{v\in \mathcal C} \mathbb C(C_v+\zeta_r^zC'_v)\subseteq V\oplus V'.$$ 
\begin{lem}\label{inva}
   The subspace $W$ is invariant by the restriction of $\psi$ to $G(r,(\tau,\tau))$.
\end{lem}
\begin{proof}
   It is clear that if $v\in \mathcal C $ and $g\in G(r,(\tau,\tau))$, then $|g|v|g|^{-1}\in \mathcal C$.  We observe that, by definition, $|g|$ permutes the elements in $\tilde v$ having the same color, and in particular $z_1(\tilde v)=z_{|g|^{-1}(1)}(\tilde v)$. Moreover, by definition, we also have $z_1(\tilde v)=z$. In particular, if $g\in G(r,2,n)\cap G(r,(\tau,\tau))$ we have
\begin{eqnarray*}
   g\cdot (C_v+\zeta_r^{z}C'_v)&=&\zeta_r^{\langle g,\tilde v\rangle}\zeta_r^{z_1(\tilde v)-z_{|g|^{-1}(1)}(\tilde v)}C_{|g|v|g|^{-1}}+\zeta_r^{z}\zeta_r^{\langle g,\tilde v\rangle}C'_{|g|v|g|^{-1}}\\
&=& \zeta_r^{\langle g,\tilde v\rangle}(C_{|g|v|g|^{-1}}+\zeta_r^{z}C'_{|g|v|g|^{-1}}),
\end{eqnarray*}
and if $g\in G(r,(\tau,\tau)$ but $g\notin G(r,2,n)$ we have
\begin{eqnarray*}
   g\cdot (C_v+\zeta_r^{z}C'_v)&=&\zeta_r^{\langle g,\tilde v\rangle}\zeta_r^{z_1(\tilde v)} C'_{|g|v|g|^{-1}}+\zeta_r^{z}\zeta_r^{\langle g,\tilde v\rangle}\zeta_r^{-z_{|g|^{-1}(1)}(\tilde v)}C_{|g|v|g|^{-1}}\\
&=& \zeta_r^{\langle g,\tilde v\rangle}(\zeta_r^{z}C'_{|g|v|g|^{-1}}+C_{|g|v|g|^{-1}}).
\end{eqnarray*}
The proof is now complete.
\end{proof}
\begin{lem}We have $V\oplus V'=\Ind_{G(r,(\tau,\tau))}^{G(r,n)}W$.   \end{lem}
\begin{proof}
 For this we need to prove that 
\begin{equation} \label{definduct}V\oplus V'=\bigoplus_{g\in K}g\cdot W,\end{equation}

where $K$ is any set of coset representatives of $G(r,(\tau,\tau))$ in $G(r,n)$. But
$$
[G(r,n):G(r,(\tau,\tau))]=\frac{n!r^n}{(\tau!)^2r^n}=\frac{n!}{(\tau !)^2},
$$
where $\tau !=t_0!\cdots t_{r'-1}!$. Moreover
$\dim(V\oplus V')=2\dim V=2 \frac{\binom{n}{2}\binom{n-2}{2}\cdots \binom{n}{2}2^{n'}
}{\tau!2}=\frac{n!}{\tau!}$
and
$\dim W=\tau!$, so that $
[G(r,n):G(r,(\tau,\tau))]=\frac{\dim (V\oplus V')}{\dim W},
$
and hence to prove \eqref{definduct} it is enough to show that $V\oplus V'\subset G(r,n)W$, which can proved by simply showing that  $C_v\in G(r,n)W$ for some $C_v$. To show this we take $\sigma=[n'+1,n'+2,\ldots,n,1,2,\ldots,n']\in S_n$. Then it follows that conjugation by $\sigma$ stabilizes $\mathcal C$, although $\sigma\notin G(r,(\tau,\tau))$. Then we have 
\begin{eqnarray*}
   \sigma\cdot (\zeta_r^zC_v+C'_v)&=&\zeta_r^zC_{v^\sigma}+\zeta_r^{z-z_{n'+1}(\tilde v)}C_{v^\sigma}'\\
&=& \zeta_r^zC_{v^\sigma}+\zeta_r^{r'(\tilde v)}C_{v^\sigma}'\\
&=& \zeta_r^zC_{v^\sigma}-C_{v^\sigma}'.
\end{eqnarray*} 
Since also $\zeta_r^zC_{v^\sigma}+C_{v^\sigma}'\in W$ we conclude that both $C_{v^\sigma}$ and $C_{v^\sigma}'$ belong to $G(r,n)W$ and the proof is complete.
\end{proof}
We are now ready to prove the main result of this section.
\begin{thm} \label{mainanticlass} Let $\GCD(p,n)=2$ and $[\tau]=[t_0,\ldots,t_{r'}]\in AC(r,p,n)^*$. Then
$$
M(c^{1}_{[\tau]})=\bigoplus_{\substack{[\lambda^{(0)},\ldots,\lambda^{(r')}]\in \Fer(r',1,p',n'):\\|\lambda^{(i)}|=t_i\,\forall i\in [0,r'-1]}}\rho^{1}_{[\lambda^{(0)},\ldots,\lambda^{(r')},\lambda^{(0)},\ldots,\lambda^{(r')}]}.
$$
\end{thm}
If $p=2$ we need to give a closer look at the $G(r,(\tau,\tau))$-representation $W$. From the proof of Lemma \ref{inva} we have  
that $gD_v=\zeta_r^{\langle g,v\rangle} D_{|g|v|g|^{-1}}=\prod _i\zeta_r^{iz(g_i)}D_{|g|v|g|^{-1}}$, where $D_v\eqdef C_v+\zeta_r^{z}C'_v$ for $v\in \mathcal C$ are the basis elements of $W$. In particular, the conditions of Proposition \ref{crit} are satisfied and the result is a straghtforward consequence of Theorem \ref{AsymGrpn}. 

If $p>2$ we simply have to observe that the $G(r,p,n)$-module $M(c^1_{[\tau]})$ is a quotient of the restriction to $G(r,p,n)$ of the $G(r,2,n)$-module $M(c^1_{\tau})$. Since $\GCD(p,n)=2$, the irreducible representations of $G(r,2,n)$ restricted to $G(r,p,n)$ remain irreducible (and are indexed in the ``same'' way). The result is then a consequence of the case $p=2$ and Theorem \ref{AsymGrpn}.

\section{The symmetric classes}\label{Sym}

In this section we complete our discussion with the description of the $G(r,p,n)$-module $M(c)$, where $c$ is any $S_n$-conjugacy class of symmetric absolute involutions in $G(r,p,n)^*$. Despite the case $p=1$ considered in \cite{CaFu} and the case of antisymmetric classes treated in \S \ref{anclas}, where a self-contained proof of the irreducible decomposition of the module $M(c)$ was given, we will need here to make use of all the main results that we have  obtained so far, namely the construction of the complete Gelfand model in \cite{Ca2}, the study of the submodules $M(c)$ for wreath products \cite{CaFu}, as well as the  discussion of the antisymmetric submodule in \S \ref{Asym}.

We first observe that, by Theorems \ref{modello} and \ref{AsymGrpn} the symmetric submodule has the following decomposition into irreducible
$$
M^0\cong\bigoplus_{[\lambda]\in \Fer(r,1,p,n)}\rho_{[\lambda]}^0.
$$

If $v$ is a symmetric absolute involution in $G(r,p,n)^*$  we denote by $\Sh(v)$ the element of $\Fer(r,1,p,n)$ which is the shape of the multitableaux of the image of $v$, via the projective Robinson-Schensted correspondence. Namely, we let
$$\Sh(v)\eqdef[\lambda]\in \Fer(r,1,p,n),$$
 where
$$v\stackrel{RS}{\longrightarrow}([P],[P]),$$
with $[P]\in \St_{[\lambda]}$.
For notational convenience, if $c$ is a $S_n$-conjugacy class of symmetric absolute involutions in $G(r,p,n)^*$ we also let $\Sh(c)=\cup_{v\in c}\Sh(v)\subset \Fer(r,1,,n)$.

We are now ready to state the main result of this section.
\begin{thm}\label{mainsym}
Let $c$ be a $S_n$-conjugacy class of symmetric absolute involutions in $G(r,p,n)^*$ and $\GCD(p,n)=1,2$. Then
the following decomposition holds:
$$M(c) \cong \bigoplus_{[\lambda]\in \Sh(c)}\rho^0_{[\lambda]}.$$
\end{thm}

Before proving this theorem we need some further preliminary observations. Fix an arbitrary $S_n$-conjugacy class $c$ of symmetric absolute involutions in $G(r,p,n)^*$, and let $c_1,\ldots,c_s$ be the $S_n$-conjugacy classes of $G(r,n)$ which are lifts of $c$ in $G(r,n)$ (one may observe that $s$ can be either $p$ or $p/2$, though this is not needed).
We will need to consider the following restriction to $G(r,p,n)$ of the submodule of the Gelfand model for $G(r,n)$ associated to the classes $c_1,\ldots,c_s$,
$$
\tilde M(c)\eqdef(M(c_1)\oplus\cdots\oplus M(c_s))\downarrow_{G(r,p,n)}
$$

Now the crucial observation is the following
\begin{lem}\label{quotient}
   The $G(r,p,n)$-module $M(c)$ is a quotient (and hence is isomorphic to a subrepresentation) of $\tilde M(c)$.
\end{lem}
\begin{proof}
Let $K(c)$ be the vector subspace of $\tilde M(c)$ spanned by the elements $C_v-C_{\zeta^{r/p}v}$ as $v$ varies among all elements in $c_1,\ldots,c_s$. Then it is clear that, as a vector space,  $M(c)$ is the quotient of   $\tilde M(c)$ by the vector subspace $K(c)$. Moreover, if $g\in G(r,p,n)$ then
\begin{eqnarray*}
\varrho(g)(C_v-C_{\zeta^{r/p}v})&=&\zeta_r^{\langle g,v\rangle} (-1)^{\inv_v(g)}C_{|g|v|g|^{-1}}-\zeta_r^{\langle g,\zeta^{r/p}v\rangle} (-1)^{\inv_{\zeta^{r/p}v}(g)}C_{|g|\zeta^{r/p}v|g|^{-1}}\\
&=& \zeta_r^{\langle g,v\rangle} (-1)^{\inv_v(g)}(C_{|g|v|g|^{-1}}-C_{\zeta^{r/p}|g|v|g|^{-1}}),
\end{eqnarray*}
since $g\in G(r,p,n)$ implies $\langle g,\zeta^{r/p}v\rangle=\langle g,v\rangle$ .
In particular we deduce that $K(c)$ is also a submodule of $\tilde M(c)$ (as $G(r,p,n)$-modules). The fact that $M(c)\cong \tilde M(c)/K(c)$ is now a direct consequence of the definition of the structures of $G(r,p,n)$-modules.
\end{proof}
We are now ready to complete the proof of the main result of this section.

\vspace{2mm}
\noindent \emph{Proof of Theorem \ref{mainsym}.}
By Theorem \ref{wreath} the $G(r,n)$-module $M(c_1)\oplus\cdots\oplus M(c_s)$ is the sum of all representations $\rho_{\lambda}$ with $\lambda\in \Sh(c_i)$, for some $i\in [s]$ or, equivalently, with $[\lambda]\in \Sh(c)$.
It follows that the restriction $\tilde M(c)$ of this representation to $G(r,p,n)$ has the following decomposition
\begin{equation}\label{tildem}
\tilde M(c)\cong \bigoplus_{\substack{[\lambda]\in \Sh(c):\\m_p(\lambda)=1}}(\rho^0_{[\lambda]})^{\oplus p}\oplus \bigoplus_{\substack{[\lambda]\in \Sh(c):\\m_p(\lambda)=2}}(\rho_{[\lambda]}^0\oplus \rho_{[\lambda]}^1 )^{\oplus p/2}.
\end{equation}
We recall that $M(c)$ is a submodule of a Gelfand model for $G(r,p,n)$ and also a submodule of $\tilde M(c)$ by Lemma \ref{quotient}, and hence, by Equation \eqref{tildem}, we have that $M(c)$ is isomorphic to a subrepresentation of
$$
\bigoplus_{\substack{[\lambda]\in \Sh(c):\\m_p(\lambda)=1}}\rho^0_{[\lambda]}\oplus \bigoplus_{\substack{[\lambda]\in \Sh(c):\\m_p(\lambda)=2}}(\rho_{[\lambda]}^0\oplus \rho_{[\lambda]}^1 ).
$$
Nevertheless, we already know that the split representations $\rho_{[\lambda]}^1$ appear in the antisymmetric submodule and so they can not appear in $M(c)$. For completing the proof  it is now sufficient to observe that, if $c$ and $c'$ are two distinct  $S_n$-conjugacy classes of symmetric absolute involutions in $G(r,p,n)^*$, then the two sets $\Sh(c)$ and $\Sh(c')$ are disjoint.\hfill $\Box$

\vspace{2mm}We can also give an explicit combinatorial description of the set $\Sh(c)$ for a given $S_n$-conjugacy class of symmetric absolute involutions in $G(r,p,n)^*$. For this we let $SC(r,n)=\{(f_0,\ldots,f_{r-1},q_0,\ldots,q_{r-1})\in \mathbb N^{2r}:\,f_0+\cdots+f_{r-1}+2(q_0+\cdots+q_{r-1})=n\}$. The set $SC(r,n)$ has already been used in \cite[\S6]{CaFu} to parametrize the $S_n$-conjugacy classes of absolute involutions in $G(r,n)$. Let $\gamma$ be the permutation of $SC(r,n)$ defined by $$\gamma(f_0,\ldots,f_{r-1},q_0,\ldots,q_{r-1})=(f_{r/p},f_{1+r/p}\ldots,f_{r-1+r/p},q_{r/p},q_{1+r/p}\ldots,q_{r-1+r/p}),$$ where the indices must be intended as elements in $\mathbb Z_r$. We observe that $\gamma$ has order $p$ and so we have an action of the cyclic group $C_p$ generated by $\gamma$ on $SC(r,n)$.  We denote the quotient set by $SC(r,p,n)^*$.
The set $SC(r,p,n)^*$ parametrizes  the $S_n$-conjugacy classes of symmetric absolute involutions in $G(r,p,n)^*$ in the following way. Let $v\in I(r,p,n)^*$ be symmetric and $\tilde v$ be any lift of $v$ in $I(r,n)$. Then the \emph{type} of $v$ is given by
$$
[f_0(\tilde v),\ldots, f_{r-1}(\tilde v),q_0(\tilde v),\ldots, q_{r-1}(\tilde v)]\in SC(r,p,n)^*,
$$
where
\begin{eqnarray*}
f_{i}(\tilde v)&=&|\{j\in[n]:\tilde v_j=j^i\}|\\
q_i(\tilde v)&=&|\{(h,k):1\leq h<k \leq n,\,\tilde v_h=k^i \textrm{ and } \tilde v_k=h^i\}|.
\end{eqnarray*}
It is clear that this is well-defined and we have that two symmetric elements in $I(r,p,n)^*$ are $S_n$-conjugate if and only if they have the same type (see also \cite[\S6]{CaFu} for the special case $p=1$).

By \cite[Proposition 3.1]{CaFu} we can now conclude that if $[\nu]=[f_0,\ldots,f_{r-1},q_0,\ldots,q_{r-1}]\in SC(r,p,n)^*$ and $c=\{v\in I(r,p,n)^*:\textrm{ $v$ is symmetric of type }[\nu]\}$, then
$$
\Sh(c)=\left\{\begin{array}{l}[\lambda^{(0)},\ldots,\lambda^{(r-1)}]\in \Fer(r,p,n)^*:\,\textrm{ for all }i\in[0,r-1],\\|\lambda_i|=f_i+2q_i\textrm{ and }\lambda^{(i)} \textrm{ has exactly } f_i \textrm{ columns of odd length}
\end{array}\right\}.
$$
For example, if we consider $v=[1^0,3^1,2^1,4^1,5^1,7^2,6^2,8^3,10^4,9^4,11^4,12^4,14^5,13^5,]\in G(6,6,14)^*$, then the type of $v$ is $[\nu]=[1,2,0,1,2,0;0,1,1,0,1,1]$. Therefore if $c$ is the $S_n$-conjugacy class of $v$ in $G(6,6,14)^*$ we have that $\Sh(c)$ is given by all elements $[\lambda^{(0)},\ldots,\lambda^{(5)}]\in \Fer(6,1,6,14)$ such that $\lambda^{(0)}$ and $\lambda^{(3)}$ have 1 box (and 1 column of odd length), $\lambda^{(1)}$ and $\lambda^{(4)}$ have 4 boxes and 2 columns of odd length,  $\lambda^{(2)}$ and $\lambda^{(5)}$ have 2 boxes and no columns of odd length, i.e.
\setlength{\unitlength}{6pt}
$$\Sh(c)=\left\{
\left[\begin{picture}(14,2.5)(0,1)
\put(0,3){\line(1,0){1}}\put(2,3){\line(1,0){3}}\put(6,3){\line(1,0){1}}\put(8,3){\line(1,0){1}}\put(10,3){\line(1,0){2}}\put(13,3){\line(1,0){1}}\put(0,2){\line(1,0){1}}\put(2,2){\line(1,0){3}}\put(6,2){\line(1,0){1}}\put(8,2){\line(1,0){1}}\put(10,2){\line(1,0){2}}\put(13,2){\line(1,0){1}}\put(2,1){\line(1,0){1}}\put(6,1){\line(1,0){1}}\put(10,1){\line(1,0){1}}\put(13,1){\line(1,0){1}}\put(10,0){\line(1,0){1}}\put(0,2){\line(0,1){1}}\put(1,2){\line(0,1){1}}\put(2,1){\line(0,1){2}}\put(3,1){\line(0,1){2}}\put(4,2){\line(0,1){1}}\put(5,2){\line(0,1){1}}\put(6,1){\line(0,1){2}}\put(7,1){\line(0,1){2}}\put(8,2){\line(0,1){1}}\put(9,2){\line(0,1){1}}\put(10,0){\line(0,1){3}}\put(11,0){\line(0,1){3}}\put(12,2){\line(0,1){1}}\put(13,1){\line(0,1){2}}\put(14,1){\line(0,1){2}}\put(1,2){,}\put(5,2){,}\put(7,2){,}\put(9,2){,}\put(12,2){,}\end{picture}\right],
\left[\begin{picture}(13,2.5)(0,1)\put(0,3){\line(1,0){1}}\put(2,3){\line(1,0){2}}\put(5,3){\line(1,0){1}}\put(7,3){\line(1,0){1}}\put(9,3){\line(1,0){2}}\put(12,3){\line(1,0){1}}\put(0,2){\line(1,0){1}}\put(2,2){\line(1,0){2}}\put(5,2){\line(1,0){1}}\put(7,2){\line(1,0){1}}\put(9,2){\line(1,0){2}}\put(12,2){\line(1,0){1}}\put(2,1){\line(1,0){1}}\put(5,1){\line(1,0){1}}\put(9,1){\line(1,0){1}}\put(12,1){\line(1,0){1}}\put(2,0){\line(1,0){1}}\put(9,0){\line(1,0){1}}\put(0,2){\line(0,1){1}}\put(1,2){\line(0,1){1}}\put(2,0){\line(0,1){3}}\put(3,0){\line(0,1){3}}\put(4,2){\line(0,1){1}}\put(5,1){\line(0,1){2}}\put(6,1){\line(0,1){2}}\put(7,2){\line(0,1){1}}\put(8,2){\line(0,1){1}}\put(9,0){\line(0,1){3}}\put(10,0){\line(0,1){3}}\put(11,2){\line(0,1){1}}\put(12,1){\line(0,1){2}}\put(13,1){\line(0,1){2}}\put(1,2){,}\put(4,2){,}\put(6,2){,}\put(8,2){,}\put(11,2){,}\end{picture}\right],
\left[\begin{picture}(15,2.5)(0,1)\put(0,3){\line(1,0){1}}\put(2,3){\line(1,0){3}}\put(6,3){\line(1,0){1}}\put(8,3){\line(1,0){1}}\put(10,3){\line(1,0){3}}\put(14,3){\line(1,0){1}}\put(0,2){\line(1,0){1}}\put(2,2){\line(1,0){3}}\put(6,2){\line(1,0){1}}\put(8,2){\line(1,0){1}}\put(10,2){\line(1,0){3}}\put(14,2){\line(1,0){1}}\put(2,1){\line(1,0){1}}\put(6,1){\line(1,0){1}}\put(10,1){\line(1,0){1}}\put(14,1){\line(1,0){1}}\put(0,2){\line(0,1){1}}\put(1,2){\line(0,1){1}}\put(2,1){\line(0,1){2}}\put(3,1){\line(0,1){2}}\put(4,2){\line(0,1){1}}\put(5,2){\line(0,1){1}}\put(6,1){\line(0,1){2}}\put(7,1){\line(0,1){2}}\put(8,2){\line(0,1){1}}\put(9,2){\line(0,1){1}}\put(13,2){\line(0,1){1}}\put(10,1){\line(0,1){2}}\put(11,1){\line(0,1){2}}\put(12,2){\line(0,1){1}}\put(14,1){\line(0,1){2}}\put(15,1){\line(0,1){2}}\put(1,2){,}\put(5,2){,}\put(7,2){,}\put(9,2){,}\put(13,2){,}\end{picture}\right]
\right\}
$$

Therefore we have the following decomposition of $M(c)$ into irreducible representations
\setlength{\unitlength}{4pt}
$$M(c)\cong
 \rho_{\left[\begin{picture}(14,2)(0,1.3)
\put(0,3){\line(1,0){1}}\put(2,3){\line(1,0){3}}\put(6,3){\line(1,0){1}}\put(8,3){\line(1,0){1}}\put(10,3){\line(1,0){2}}\put(13,3){\line(1,0){1}}\put(0,2){\line(1,0){1}}\put(2,2){\line(1,0){3}}\put(6,2){\line(1,0){1}}\put(8,2){\line(1,0){1}}\put(10,2){\line(1,0){2}}\put(13,2){\line(1,0){1}}\put(2,1){\line(1,0){1}}\put(6,1){\line(1,0){1}}\put(10,1){\line(1,0){1}}\put(13,1){\line(1,0){1}}\put(10,0){\line(1,0){1}}\put(0,2){\line(0,1){1}}\put(1,2){\line(0,1){1}}\put(2,1){\line(0,1){2}}\put(3,1){\line(0,1){2}}\put(4,2){\line(0,1){1}}\put(5,2){\line(0,1){1}}\put(6,1){\line(0,1){2}}\put(7,1){\line(0,1){2}}\put(8,2){\line(0,1){1}}\put(9,2){\line(0,1){1}}\put(10,0){\line(0,1){3}}\put(11,0){\line(0,1){3}}\put(12,2){\line(0,1){1}}\put(13,1){\line(0,1){2}}\put(14,1){\line(0,1){2}}\put(1,2){,}\put(5,2){,}\put(7,2){,}\put(9,2){,}\put(12,2){,}\end{picture}\right]}
\oplus
\rho^0_{\left[\begin{picture}(13,2)(0,1.3)\put(0,3){\line(1,0){1}}\put(2,3){\line(1,0){2}}\put(5,3){\line(1,0){1}}\put(7,3){\line(1,0){1}}\put(9,3){\line(1,0){2}}\put(12,3){\line(1,0){1}}\put(0,2){\line(1,0){1}}\put(2,2){\line(1,0){2}}\put(5,2){\line(1,0){1}}\put(7,2){\line(1,0){1}}\put(9,2){\line(1,0){2}}\put(12,2){\line(1,0){1}}\put(2,1){\line(1,0){1}}\put(5,1){\line(1,0){1}}\put(9,1){\line(1,0){1}}\put(12,1){\line(1,0){1}}\put(2,0){\line(1,0){1}}\put(9,0){\line(1,0){1}}\put(0,2){\line(0,1){1}}\put(1,2){\line(0,1){1}}\put(2,0){\line(0,1){3}}\put(3,0){\line(0,1){3}}\put(4,2){\line(0,1){1}}\put(5,1){\line(0,1){2}}\put(6,1){\line(0,1){2}}\put(7,2){\line(0,1){1}}\put(8,2){\line(0,1){1}}\put(9,0){\line(0,1){3}}\put(10,0){\line(0,1){3}}\put(11,2){\line(0,1){1}}\put(12,1){\line(0,1){2}}\put(13,1){\line(0,1){2}}\put(1,2){,}\put(4,2){,}\put(6,2){,}\put(8,2){,}\put(11,2){,}\end{picture}\right]}
\oplus
\rho^0_{\left[\begin{picture}(15,2)(0,1.3)\put(0,3){\line(1,0){1}}\put(2,3){\line(1,0){3}}\put(6,3){\line(1,0){1}}\put(8,3){\line(1,0){1}}\put(10,3){\line(1,0){3}}\put(14,3){\line(1,0){1}}\put(0,2){\line(1,0){1}}\put(2,2){\line(1,0){3}}\put(6,2){\line(1,0){1}}\put(8,2){\line(1,0){1}}\put(10,2){\line(1,0){3}}\put(14,2){\line(1,0){1}}\put(2,1){\line(1,0){1}}\put(6,1){\line(1,0){1}}\put(10,1){\line(1,0){1}}\put(14,1){\line(1,0){1}}\put(0,2){\line(0,1){1}}\put(1,2){\line(0,1){1}}\put(2,1){\line(0,1){2}}\put(3,1){\line(0,1){2}}\put(4,2){\line(0,1){1}}\put(5,2){\line(0,1){1}}\put(6,1){\line(0,1){2}}\put(7,1){\line(0,1){2}}\put(8,2){\line(0,1){1}}\put(9,2){\line(0,1){1}}\put(13,2){\line(0,1){1}}\put(10,1){\line(0,1){2}}\put(11,1){\line(0,1){2}}\put(12,2){\line(0,1){1}}\put(14,1){\line(0,1){2}}\put(15,1){\line(0,1){2}}\put(1,2){,}\put(5,2){,}\put(7,2){,}\put(9,2){,}\put(13,2){,}\end{picture}\right]}
$$

\section{A final survey and a further generalization}\label{last}

The aim of this section is to provide a statement containing all the information we collected about the refinement of the model in the rest of the paper and holding in full generality for all the groups $G(r,p,n)$ with $\GCD(p,n)=1,2$.
Furthermore, we use this occasion to observe that the above results apply, in fact, to  all the projective groups $G(r,p,q,n)$ with $\GCD(p,n)=1,2.$

If $v$ is an absolute involution in $G(r,q,p,n)$ and $c$ is any (symmetric or antisymmetric) $S_n$-conjugacy class of absolute involutions in $G(r,q,p,n)$, we define $\Sh(v)\in \Fer(r,q,p,n)$ and $\Sh(c)\subset \Fer(r,q,p,n)$ as in Section \ref{Sym}. Moreover, we let $\iota(c)=0$ if the elements of $c$ are symmetric and $\iota(c)=1$ if the elements of $c$ are antisymmetric.

\begin{thm} \label{verymain}Let $G=G(r,p,q,n)$ with $\GCD(p,n)=1,2$, and let $(M(r,q,p,n), \varrho)$ be its Gelfand model defined in Theorem \ref{modello}. Given an $S_n$-conjugacy class $c$ of absolute involutions in $G^*$, let $M(c)={\rm Span}\{C_v:\,v\in c\}$  so that $M(r,q,p,n)$ naturally splits as a $G$-module into the direct sum
$$M(r,q,p,n)=\bigoplus_c M(c).$$
Then the submodule $M(c)$ has the following decomposition into irreducibles
$$
M(c)\cong \bigoplus_{[\lambda]\in \Sh(c)}\rho_{[\lambda]}^{\iota(c)}.
$$
\end{thm}

\begin{proof} We have already established this result if $q=1$. In fact, if $\iota(c)=0$ this is the content of Theorem \ref{mainsym}, and, if $\iota(c)=1$, the result follows directly from Theorem \ref{mainanticlass} with the further observation that if $v$ is an antisymmetric element of type $[t_0,\ldots,t_{r'-1}]$, then $\Sh(v)=[\lambda^{(0)},\ldots,\lambda^{(r-1)}]$, with $|\lambda^{(i)}|= |\lambda^{(i+r')}|=t_i$. 

If $q\neq 1$ the result is straightforward since an $S_n$-conjugacy class of absolute involutions in $G(r,q,p,n)$ is also an $S_n$-conjugacy class of absolute involutions in $G(r,1,p,n)$ and  the definition of the Gelfand models for $G(r,p,q,n)$ and $G(r,p,1,n)$ are compatible with the projection $G(r,p,1,n)\rightarrow G(r,p,q,n)$. 
\end{proof}

\end{document}